\documentclass[12pt]{amsart}
\usepackage{amssymb,amscd,amsthm,amsmath,color}
\usepackage{fullpage}
\usepackage{epsfig}
\usepackage{graphicx}
\usepackage{xypic}
\usepackage{url}
\usepackage{mathtools}
\usepackage{hyperref}
\usepackage{mathrsfs}
\usepackage{quiver}
\usepackage{enumitem}
\usepackage[style=ieee-alphabetic, isbn=false, eprint=false, url=false, backref=false, doi=false]{biblatex}
\numberwithin{equation}{section}

\makeatletter
\let\@wraptoccontribs\wraptoccontribs
\makeatother

\newcommand{\PP}{\mathbb{P}}

\newcommand{\fp}{\mathfrak{p}}

\newcommand{\Spec}{\operatorname{Spec}}

\newcommand{\rk}{\operatorname{rk}}

\newcommand{\Eff}{\overline{\operatorname{Eff}}}
\newcommand{\Nef}{\operatorname{Nef}}
\newcommand{\Supp}{\operatorname{Supp}}
\newcommand{\Mor}{\operatorname{Mor}}

\newcommand{\Pic}{\operatorname{Pic}}

\newtheorem{theorem}{Theorem}[section]
\newtheorem{lemma}[theorem]{Lemma}
\newtheorem{proposition}[theorem]{Proposition}

\theoremstyle{definition}
\newtheorem{definition}[theorem]{Definition}
\newtheorem{notation}[theorem]{Notation}
\newtheorem{conjecture}[theorem]{Conjecture}
\newtheorem{remark}[theorem]{Remark}

\newtheorem{example}[theorem]{Example}

\definecolor{ao(english)}{rgb}{0.0, 0.5, 0.0}

\input xy
\xyoption{all}

\makeatletter
\def\dasharrowfill@#1#2#3#4{%
        $\m@th
        \thickmuskip0mu
        \medmuskip\thickmuskip
        \thinmuskip\thickmuskip
        \relax
        #4#1\mkern2mu
        \xleaders\hbox{$#4\mkern2mu#2\mkern2mu$}\hfill
        \mkern2mu
        #3$%
}

\def\dashrightarrowfill@{\dasharrowfill@\relbar\relbar\rightarrow}

\providecommand*\xdashrightarrow[2][]{%
  \ext@arrow 0055{\dashrightarrowfill@}{#1}{#2}}
\makeatother

\begin{document}

\title{Free curves and fundamental groups}

\author{Eric Jovinelly}
\address{Department of Mathematics \\
Brown University \\
Box 1917 \\
151 Thayer Street \\
Providence, RI, 02912}
\email{eric\_jovinelly@brown.edu}

\author{Brian Lehmann}
\address{Department of Mathematics \\
Boston College  \\
Chestnut Hill, MA 02467}
\email{lehmannb@bc.edu}

\author{Eric Riedl}
\address{Department of Mathematics \\
University of Notre Dame  \\
255 Hurley Hall \\
Notre Dame, IN 46556}
\email{eriedl@nd.edu}

\contrib[with an appendix by]{Aise Johan de Jong}
\address{Department of Mathematics \\
Columbia University  \\
New York, NY 10027}
\email{dejong@math.columbia.edu}

\begin{abstract}
We show that klt Fano varieties and certain lc Fano varieties contain free higher-genus curves in their smooth loci.  
Our methods also allow us to find free curves on varieties in positive characteristic and on quasiprojective varieties, under a natural positivity condition on the tangent bundle.  We then use the existence of free curves to deduce finiteness of the fundamental group of the smooth locus in these settings.

The paper includes an appendix by de Jong that establishes the K{\"u}nneth formula for tame \'etale fundamental groups.
\end{abstract}

\maketitle

\section{Introduction}

Rational curves provide an essential tool for studying the geometry of a smooth complex Fano variety $X$.  A paragon is \cite{Mori79} in which Mori proved Hartshorne's Conjecture by studying the existence and deformation properties of rational curves.  A rational curve $s: \mathbb{P}^{1} \to X$ exhibits the best deformation-theoretic properties when $s^{*}T_{X}$ is positive: the map $s$ is said to be free if $s^{*}T_{X}$ is nef and very free if $s^{*}T_{X}$ is ample.  Using the innovative techniques of \cite{Mori79}, Koll\'ar showed that every smooth Fano variety carries a free rational curve (see \cite[Page 749]{Mori84}, \cite[Page 66]{MM86}).  The existence of very free curves on smooth Fano varieties was subsequently established by \cite{KMM92}.

When $X$ is a singular Fano variety, the situation is more subtle.  It is still true that $X$ carries many rational curves; for example, \cite{HM07} and \cite{Zhang06} prove that a klt Fano pair $(X,\Delta)$ admits a rational curve through two general points.  However, these rational curves can meet the singular locus of $X$ in which case they no longer exhibit good deformation-theoretic properties.  For potential applications it is crucial to find rational curves in the smooth locus of $X$.

The following well-known conjecture is motivated by work of Keel-M\textsuperscript{c}Kernan and Miyanishi-Tsunoda.  

\begin{conjecture} \label{conj:kltfano}
Let $(X,\Delta)$ be a klt Fano pair over an algebraically closed field of characteristic $0$.  Then there is a free rational curve in the smooth locus of $X$.
\end{conjecture}

Building on the earlier works  \cite{MT84a, MT84b, Zhang88}, \cite[1.6 Corollary]{KM99} verified this conjecture in dimension $2$: a klt Fano surface always admits a free rational curve in its smooth locus.  This result is a central step in Keel and M\textsuperscript{c}Kernan's groundbreaking work on the existence of $\mathbb{A}^{1}$-curves and also has a number of important applications in its own right.   In the decades since, the question of existence of rational curves in the smooth locus of a Fano variety has remained open, even for terminal Fano threefolds.  

\subsection{Main results: characteristic 0}
We take a novel approach: we relax the requirement that our curves be rational and instead search for curves of arbitrary genus.  The following definition identifies a natural analogue of freeness for higher genus curves.

\begin{definition} \label{defi:freecurve}
Let $X$ be a variety and let $C$ be a smooth proper integral curve.  Fix a non-negative number $r$.  A morphism $s: C \to X$ is an $r$-free curve if $s(C)$ is contained in the smooth locus $X^{sm} \subset X$ and every positive rank quotient of $s^{*}T_{X}$ has slope at least $2g(C) + r$.
\end{definition}

With this change in perspective, we are able to prove the analogue of Conjecture \ref{conj:kltfano} for higher genus curves.

\begin{theorem} \label{theo:dltcase}
Let $(X,\Delta)$ be a dlt Fano pair over an algebraically closed field of characteristic $0$.  Then there is a curve $C$ such that for every $r \geq 0$ there is an $r$-free curve $s: C \to X^{sm}$.
\end{theorem}

Using this result and \cite{JLR25}, we take the next step toward the original Conjecture \ref{conj:kltfano} in higher dimensions.

\begin{theorem} \label{theo:terminalcase}
Let $X$ be a terminal Fano threefold over an algebraically closed field of characteristic $0$.  Then there is a free rational curve $s: \mathbb{P}^{1} \to X^{sm}$.
\end{theorem}

Our techniques are applicable to lc Fano pairs $(X,\Delta)$.  Note that such an $X$ need not contain any free curves at all.  (The simplest counterexample is a cone over an elliptic curve.)  Instead, we show that the existence of free curves can be precisely accounted for by the positivity of the tangent bundle.  Furthermore, this class of lc Fano pairs shares many of the nice geometric properties of klt Fano pairs; see Section \ref{sect:dichotomy}.

\begin{theorem} \label{theo:lccase}
Let $(X,\Delta)$ be a $\mathbb{Q}$-factorial lc Fano pair over an algebraically closed field of characteristic $0$.  Then the following conditions are equivalent:
\begin{enumerate}
\item $X$ admits a $1$-free curve.
\item There is no positive rank quotient of $T_{X}$ with numerically trivial first Chern class. 
\item For every $m > 0$ we have $h^{0}(X^{sm},\Omega_{X^{sm}}^{\otimes m}) = 0$.
\end{enumerate}
If these equivalent conditions hold, then there is a curve $C$ such that for every $r \geq 0$ there is an $r$-free curve $s: C \to X$.
\end{theorem}

Just as with free rational curves, the existence of a $1$-free curve of arbitrary genus has a number of important geometric consequences.  The most significant is finiteness of the fundamental group.  Recently there have been many results concerning fundamental groups associated to klt singularities; see e.g.~\cite{Xu14}, \cite{GKP16a}, \cite{BGO17}, \cite{TX17}, \cite{CRST18}, \cite{BCRGST19}, \cite{Braun21}, \cite{CarvajalStabler}.  In combination with Theorem \ref{theo:lccase}, the following result extends our understanding to the log canonical case.

\begin{theorem} \label{theo:firstfundgroups}
Let $X$ be a projective variety over $\mathbb{C}$.  If $X$ admits a $1$-free curve, then $X^{sm}$ has finite topological fundamental group.
\end{theorem}

In particular, we obtain a new (and short) proof of the finiteness of the topological fundamental group of the smooth locus of a klt Fano pair first proved by \cite{Braun21}.

\subsection{Main results: characteristic $p$}

Much less is known about the existence of free curves in characteristic $p$.  A famous open problem of Koll\'ar asks whether a smooth Fano variety over an algebraically closed field of characteristic $p$ is separably rationally connected.  If we allow singularities, the answer is no: counterexamples are given in \cite{Kollar95} and \cite[Section 5]{Xu12}.  These examples use a ``negative'' quotient of the tangent bundle to obstruct the existence of free rational curves.  In particular, they demonstrate that in characteristic $p$ we must impose stronger assumptions on the possible quotients of the tangent bundle to obtain free curves. 

The following theorem provides a new link in the relationship between positivity of the tangent bundle and existence of free curves, in the same direction as earlier work by \cite{Miyaoka85, KSCT07, SCT09, Shen10, Tian15, Gounelas16, BM16, CP19} and many others.  The positivity assumption is phrased in terms of Mumford-Takemoto slope stability. 

\begin{theorem} \label{theo:maintheoremcharp}
Let $X$ be a normal projective variety of dimension $n$ over an algebraically closed field.  Suppose that there is an ample divisor $H$ such that $\mu^{min}_{H^{n-1}}(T_{X}) > 0$.
Then there is a curve $C$ such that for every $r \geq 0$ there is an $r$-free curve $s: C \to X^{sm}$. 
\end{theorem}

In characteristic $p$, this theorem is interesting even when $X$ is smooth; for example, it applies whenever $X$ is a smooth Fano variety with a stable tangent bundle.  
Just as in characteristic $0$, we can use $1$-free curves to deduce finiteness of fundamental groups.

\begin{theorem} \label{theo:etalefundgroup}
Let $X$ be a projective variety over an algebraically closed field.  Suppose that $X$ carries a $1$-free curve. Then the \'etale fundamental group of $X^{sm}$ is finite.
\end{theorem}

We emphasize that it is the \'etale fundamental group and not the tame \'etale fundamental group that appears in Theorem \ref{theo:etalefundgroup}.

\subsection{Main results: log setting}
There are several famous conjectures that predict that in the quasiprojective setting curves have similar behavior to the projective setting.  According to Iitaka's philosophy, the best way to analyze curves on a smooth quasiprojective variety $U$ is to pass to a SNC compactification $(X,\Delta)$ and study log curves.  In particular, when $(X,\Delta)$ has ``positive curvature'' we should expect to find log curves on $(X,\Delta)$ which are $1$-free in the sense of Definition \ref{defi:rfreelogcurve}.

In our setting, it is more natural to start from an lc pair $(X,\Delta)$ and to define $U$ to be the open subscheme of $X$ obtained by removing all irreducible components of $\Delta$ with coefficient $1$.  Even when $(X,\Delta)$ is an lc Fano pair, it need not admit a $1$-free log curve.  (The simplest counterexample is $\mathbb{P}^{2}$ equipped with two lines.)  As before, we show that the positivity of the log tangent bundle is the key ingredient for correcting this deficiency.

\begin{theorem} \label{theo:logcase}
Let $(X,\Delta)$ be a $\mathbb{Q}$-factorial lc Fano pair over an algebraically closed field of characteristic $0$.  Then the following conditions are equivalent:
\begin{enumerate}
\item $(X,\Delta)$ admits a $1$-free log curve.
\item The log tangent bundle admits no positive rank quotient with numerically trivial first Chern class.
\item For every $m > 0$ we have $h^{0}(V,\Omega_{V}(\log \lfloor \Delta \rfloor)^{\otimes m}) = 0$ where $V \subset X$ is the SNC locus of the pair $(X,\lfloor \Delta \rfloor)$.
\end{enumerate}
If these equivalent conditions hold, then there is a curve $C$ and a finite subset $D \subset C$ such that for every $r \geq 0$ there is an $r$-free log curve $s: (C,D) \to (X, \Delta )$.
\end{theorem}

In dimension $2$ \cite{Zhu16} proves a similar but more precise statement.  We also prove a version of Theorem \ref{theo:logcase} in positive characteristic; see Theorem \ref{theo:logcharp}.

\begin{remark}
We expect the conditions of Theorem \ref{theo:logcase} to be equivalent to $\mathbb{A}^{1}$-connectedness of the open locus.  More generally, we expect these conditions to be important in other contexts which rely on ``geometric positivity'' in the log setting, e.g.~the study of integral points for Fano pairs.
\end{remark}

In the log setting, we would like to understand the fundamental group of the open locus $U$.  Note that $U$ need not be smooth; our next result bounds the fundamental group of the smooth locus of $U$.

\begin{theorem} \label{theo:logfundgroup}
Let $(X,\Delta)$ be a projective lc pair over an algebraically closed field $k$.  Set $U = X \backslash \lfloor \Delta \rfloor$.  Suppose that $(X,\Delta)$ carries a $1$-free log curve.  Then:
\begin{enumerate}
\item The curve-tame \'etale fundamental group of $U^{sm}$ is finite.
\item If $k = \mathbb{C}$, then the topological fundamental group of $U^{sm}$ is finite.
\end{enumerate}
\end{theorem}

The proof of Theorem \ref{theo:logfundgroup} relies on Appendix \ref{sect:appendix} by de Jong which establishes the behavior of tame \'etale fundamental groups under products.

\bigskip

\noindent \textbf{Acknowledgements:}
We are grateful to Johan de Jong for writing the appendix.  We thank Adrian Langer for identifying an error in an earlier version of the paper and for an in-depth discussion about Theorem \ref{theo:normalsemistability}.  We are grateful to J\'anos Koll\'ar for explaining a simplification of the argument for Theorem \ref{theo:firstetalefundgroup}.  We also thank Nathan Chen, Izzet Coskun, Lawrence Ein, Phil Engel, Brendan Hassett, Eric Larson, Akhil Matthew, James M\textsuperscript{c}Kernan, Joaqu\'in Moraga, Nick Salter, Matthew Satriano, Sho Tanimoto, Kevin Tucker, and Isabel Vogt for helpful conversations.

Eric Jovinelly was supported by an NSF postdoctoral research fellowship, DMS-2303335. Brian Lehmann was supported by Simons Foundation grant Award Number 851129.  Eric Riedl was supported by NSF CAREER grant DMS-1945944 and Simons Foundation grants 00011850 and 00013673.

\section{Preliminaries}

Throughout we work with schemes which are separated and whose connected components have finite type over the ground field.  A variety is irreducible and reduced.

Over an algebraically closed field of arbitrary characteristic, we let $\overline{\mathcal{M}}_{g,n}(X)$ denote the Kontsevich moduli stack of stable maps and let $\mathcal{M}_{g,n}(X)$ denote the open substack parametrizing morphisms with smooth irreducible domain.  

Suppose $k$ is a field of positive characteristic $p$ and we have a $k$-scheme $g: X \to \Spec(k)$.  We denote the $p$th power absolute Frobenius map on $X$ by $F_{abs}: X \to X$.  We let $F_{rel}$ denote the relative Frobenius map over $\Spec(k)$, i.e.~if $X^{(1)}$ denotes the base change of $g: X \to \Spec(k)$ over $F_{abs}: \Spec(k) \to \Spec(k)$ then $F_{rel} = (g,F_{abs}): X \to X^{(1)}$ is a morphism over $\Spec(k)$.  We denote the $e$th iterate of $F_{rel}$ by $F^{e}_{rel}: X \to X^{(e)}$.  When the ground field is perfect, the Frobenius automorphism of $\Spec(k)$ is invertible so we can also define $F^{e}_{rel}: X^{(-e)} \to X$.

Note that when $k = \mathbb{F}_{p^{b}}$ is a finite field the homomorphism $F_{abs}^{m}: k \to k$ is the identity map for any $m$ divisible by $b$.  Thus if $m$ is a multiple of $b$ and $X$ is a $k$-scheme then $X^{(-m)}$ will be isomorphic to $X$ as a $k$-scheme.  Similarly, any divisor on $X$ defined over $k$ will be mapped to itself by $F_{abs}^{m}$.

\subsection{Non $\mathbb{Q}$-factorial varieties}

When running the MMP, it is possible to obtain a Fano fibration whose fibers are not $\mathbb{Q}$-factorial.  For this reason, we need to briefly review how to work with curves on non-$\mathbb{Q}$-factorial varieties.  Note that for such varieties the behavior of Cartier and Weil divisors can be quite different and we must be careful to choose the correct notion.

Suppose that $X$ is a normal projective variety of dimension $n$ over an algebraically closed field.  We first define spaces associated to Cartier divisors.  We let $N^{1}(X)_{\mathbb{R}}$ denote the space of $\mathbb{R}$-Cartier divisors up to numerical equivalence and let $N_{1}(X)_{\mathbb{R}}$ denote the dual space of $\mathbb{R}$-$1$-cycles up to numerical equivalence.  We let $\Eff^{1}(X) \subset N^{1}(X)_{\mathbb{R}}$ denote the pseudo-effective cone of divisors and let $\Eff_{1}(X) \subset N_{1}(X)_{\mathbb{R}}$ denote the pseudo-effective cone of curves.  Dual to these are the nef cones $\Nef^{1}(X) \subset N^{1}(X)_{\mathbb{R}}$ and $\Nef_{1}(X) \subset N_{1}(X)_{\mathbb{R}}$.

We next define spaces associated to Weil divisors.  We let $N_{n-1}(X)_{\mathbb{R}}$ denote the space of $\mathbb{R}$-Weil divisors up to intersection against homogeneous polynomials of weight $(n-1)$ in Chern classes of vector bundles on $X$.  We let $N^{n-1}(X)_{\mathbb{R}}$ denote the dual space spanned by polynomials in Chern classes.   We let $\Eff_{n-1}(X) \subset N_{n-1}(X)_{\mathbb{R}}$ denote the closure of the cone generated by effective Weil divisors.  Dually we have the nef cone $\Nef^{n-1}(X) \subset N^{n-1}(X)_{\mathbb{R}}$.

\begin{remark}
For any normal projective variety $X$ there is an injective map $N^{1}(X)_{\mathbb{R}} \xrightarrow{\cap[X]} N_{n-1}(X)_{\mathbb{R}}$ from numerical Cartier classes to numerical Weil classes.  Dually, we get a surjection from $\Nef^{n-1}(X)$ to $\Nef_{1}(X)$.  In particular, when we want to find curves representing a nef class the strongest statements are obtained by working with $\Nef^{n-1}(X)$.
\end{remark}

\begin{remark}
When $X$ is $\mathbb{Q}$-factorial, the map $N^{1}(X)_{\mathbb{R}} \xrightarrow{\cap[X]} N_{n-1}(X)_{\mathbb{R}}$ is an isomorphism.  Thus we can naturally identify $\Nef^{n-1}(X)$ and $\Nef_{1}(X)$.  We will preferentially use the latter notation when $X$ is $\mathbb{Q}$-factorial.
\end{remark}

\begin{remark} \label{rema:defineint}
Suppose that $C$ is a smooth curve and $s: C \to X$ is a map whose image lies in the smooth locus of $X$.  Note that $s$ is an lci morphism of codimension $n-1$; indeed, we can write $s$ as the regular embedding $\Gamma_{s}: C \to C \times X^{sm}$ followed by the smooth projection $\pi_{2}: C \times X^{sm} \to X$.  Then \cite[Example 19.2.3]{Fulton84} shows that the Gysin map $s^{*}: A_{n-1}(X) \to A_{0}(C)$ descends to numerical groups.  In this way, any curve $C$ in the smooth locus of $X$ naturally defines a class $s_{*}C \in N^{n-1}(X)_{\mathbb{R}}$.  If furthermore $s$ deforms in a dominant family then $s_{*}C$ defines a class in $\Nef^{n-1}(X)$.
\end{remark}

The advantage of $\Nef^{n-1}(X)$ is that it is compatible with birational pullbacks, as explained by the following lemma.

\begin{lemma} \label{lemm:nonqfactorialcone}
Let $\phi: X' \to X$ denote a birational morphism between normal projective varieties over an algebraically closed field.  Then:
\begin{enumerate}
\item The pullback map $\phi^{*}: \Nef^{n-1}(X) \to \Nef^{n-1}(X')$ is injective.
\item If $\phi$ is a small contraction, then $\phi^{*}$ takes interior classes for $\Nef^{n-1}(X)$ to interior classes for $\Nef^{n-1}(X')$. 
\item The pushforward map $\phi_{*}: N_{1}(X')_{\mathbb{R}} \to N_{1}(X)_{\mathbb{R}}$ takes $\Nef_{1}(X')$ surjectively onto $\Nef_{1}(X)$.
\end{enumerate}
\end{lemma}

\begin{proof}
(1) follows from the surjectivity of $\Eff_{n-1}(X') \xrightarrow{\phi_{*}} \Eff_{n-1}(X)$.  Furthermore, \cite[Theorem 3.21]{FL16} shows that when $\phi$ is a small contraction the kernel of $N_{n-1}(X')_{\mathbb{R}}  \xrightarrow{\phi_{*}} N_{n-1}(X)_{\mathbb{R}}$ only intersects $\Eff_{n-1}(X')$ at $\{ 0 \}$.  Dualizing, we obtain (2).  (3) is dual to the fact that the pseudo-effective cone of Cartier divisors injects upon birational pullback.
\end{proof}

\subsection{Stability of torsion-free sheaves}

We quickly review Harder-Narasimhan filtrations with respect to nef curve classes as developed by \cite{CP11,GKP14,GKP16}.  While these references work in characteristic $0$, the arguments used for the construction below work equally well in characteristic $p$. 

\begin{notation}
Let $X$ be a $\mathbb{Q}$-factorial normal projective variety and let $\alpha \in \Nef_{1}(X)$.  For any non-zero torsion-free sheaf $\mathcal{E}$ on $X$, let
\begin{equation*}
0 = \mathcal{F}_{0} \subset \mathcal{F}_{1} \subset \mathcal{F}_{2} \subset \ldots \subset \mathcal{F}_{s} = \mathcal{E}.
\end{equation*}
denote the $\alpha$-Harder-Narasimhan filtration of $\mathcal{E}$ with respect to $\alpha$.  Thus each $\mathcal{F}_{i}/\mathcal{F}_{i-1}$ is $\alpha$-semistable and the $\alpha$-slopes of the quotients $\mathcal{F}_{i}/\mathcal{F}_{i-1}$ are strictly decreasing in $i$.  When $X$ is a curve, we will always implicitly use the polarization $\alpha$ given by the fundamental class.

We denote by $\mu^{max}_{\alpha}(\mathcal{E})$ the maximal $\alpha$-slope of any torsion-free subsheaf of $\mathcal{E}$, that is, $\mu^{max}_{\alpha}(\mathcal{E}) = \mu_{\alpha}(\mathcal{F}_{1})$.  We denote by $\mu^{min}_{\alpha}(\mathcal{E})$ the minimal $\alpha$-slope of any torsion-free quotient of $\mathcal{E}$, that is, $\mu^{min}_{\alpha}(\mathcal{E}) = \mu_{\alpha}(\mathcal{E}/\mathcal{F}_{s-1})$.
\end{notation}

One can develop a theory of slope stability for arbitrary normal projective varieties by working with $\Nef^{n-1}(X)$ instead of $\Nef_{1}(X)$.  While we will not develop such a theory here, we will need one special case: when $H$ is an ample divisor on a normal projective variety, one can define a theory of slope stability with respect to $H^{n-1}$.  Indeed, this is the classical notion of Mumford-Takemoto slope stability; see \cite{Maruyama81}.

We will also need Langer's results on the behavior of semistability under restriction in characteristic $p$.  More precisely, we will need a version of \cite[Theorem 4.1]{Langer10} for normal projective varieties.

\begin{theorem}[\cite{Langer10}] \label{theo:normalsemistability}
Let $X$ be a normal projective variety of dimension $n \geq 2$ and let $H$ be a very ample Cartier divisor on $X$.  For a tuple $\vec{m} = (m_{1},\ldots,m_{n-1})$ of positive integers, let $\mathcal{C}_{\vec{m}}$ denote the generic fiber of the family of complete intersection curves constructed from $|m_{1}H|,\ldots,|m_{n-1}H|$.

Suppose that $\mu^{min}_{H^{n-1}}(\Omega_{X}^{\vee \vee}) \leq 0$.  Then for any finite set of $H^{n-1}$-semistable torsion-free sheaves $\{ \mathcal{E}_{j}\}$ there is a choice of $\vec{m}$ such that for every $j$ the restriction of $\mathcal{E}_{j}$ to $\mathcal{C}_{\vec{m}}$ is strongly semistable.
\end{theorem}

\begin{proof}
For simplicity we explain the proof for a single semistable sheaf $\mathcal{E}$; the argument is exactly the same for a finite set of sheaves.

The assumption on the minimal slope quotient of $\Omega_{X}^{\vee \vee}$ implies that each semistable sheaf on $X$ is in fact strongly semistable by applying Cartier descent to the smooth locus of $X$ as in the argument of \cite[Theorem 2.1]{MR83}.  Moreover, the argument of \cite[Corollary A.3.1]{Langer04} works for normal projective varieties to show that the torsion-free part of a tensor product of strongly semistable sheaves on $X$ is again strongly semistable.

By repeatedly applying \cite[Theorem 0.1]{Langer24}, we obtain an infinite set of tuples of positive integers $(m_{1},\ldots,m_{n-1})$ satisfying the following properties:
\begin{enumerate}
\item The restriction of $\mathcal{E}$  to $\mathcal{C}_{\vec{m}}$ is semistable.
\item The restriction of each graded piece of the $H^{n-1}$-Harder-Narasimhan filtration of $\Omega_{X}^{\vee \vee}$ is semistable.
\item There are positive constants $C_{1},C_{2}$ such that $C_{1} < \frac{m_{i}}{m_{j}} < C_{2}$ for any pair of integers $m_{i},m_{j}$ taken from a tuple in our set.
\end{enumerate}
The last property comes from the formulation of the explicit bounds on the $m_{i}$ in \cite[Theorem 0.1]{Langer24}.  Note that condition (3) implies that as we vary $\vec{m}$ in our set the positive integers $\min (\vec{m})$ are unbounded.

The remainder of the proof follows the steps of \cite[Theorem 4.1]{Langer10}: suppose that $\mathcal{E}|_{\mathcal{C}_{\vec{m}}}$ fails to be strongly semistable.  Let $Z_{\vec{m}} \subset \prod_{i=1}^{n-1} \mathbb{P}(|m_{i}H|) \times X$ denote the total incidence correspondence that generically parametrizes complete intersection curves.  By \cite[Theorem 3.1]{Langer24} there is some iterated absolute Frobenius $F^{e}_{abs}$ on $\mathcal{C}_{\vec{m}}$ such that every graded piece of the Harder-Narasimhan filtration of $F^{e*}_{abs}\mathcal{E}|_{\mathcal{C}_{\vec{m}}}$ is strongly semistable; we choose the minimal such $e$.  Letting $\pi: Z_{\vec{m}} \to X$ denote the projection, by spreading this filtration of $F^{e*}_{abs}\mathcal{E}|_{\mathcal{C}_{\vec{m}}}$ out to $Z_{\vec{m}}$ we get a filtration
\begin{equation*}
0 = \mathcal{F}_{0} \subset \mathcal{F}_{1} \subset \ldots \subset \mathcal{F}_{s} = F^{e*}_{abs}\pi^{*}\mathcal{E}
\end{equation*}
whose restriction to $\mathcal{C}_{\vec{m}}$ is the Harder-Narasimhan filtration.  Note that each torsion-free sheaf on $X$ restricts to a locally free sheaf on $\mathcal{C}_{\vec{m}}$.  As in \cite[Theorem 4.1]{Langer10}, one shows that there exists an index $i$ such that
\begin{align*}
\frac{m_{1}\ldots m_{n-1}}{\max \{ \tfrac{r^{2}-1}{4}, 1\}} & \leq \mu^{min}((\mathcal{F}_{i} \otimes (F^{e*}_{abs}\pi^{*}\mathcal{E}/\mathcal{F}_{i})^{\vee})|_{\mathcal{C}_{\vec{m}}} ) \\
& \leq \mu^{max}((\Omega_{Z_{\vec{m}}/X}|_{\mathcal{C}_{\vec{m}}})^{\vee \vee}) \leq \max_{j} \frac{m_{j} m_{1} \ldots m_{n-1} H^{n}}{ {m_{j}+n \choose n} - (n-1)m_{j} - 1 }
\end{align*}
where $r$ denotes the rank of $\mathcal{E}$.
(The numerical bounds differ from those of \cite[Theorem 4.1]{Langer10} since we are working with curves instead of divisors; the details of the computation are given in \cite[Proof of Theorem 7.1.1]{HL10}.)  According to condition (3) on our set of tuples, there are only finitely many tuples $\vec{m}$ in our set for which this inequality holds.  Thus there exists a tuple $\vec{m}$ such that $\mathcal{E}|_{\mathcal{C}_{\vec{m}}}$ is strongly semistable.
\end{proof}

\subsection{Positive curves}

The notion of an $r$-free curve was given in Definition \ref{defi:freecurve}.  Note that for a rational curve, free is the same as $0$-free and very free is the same as $1$-free; we will use the same shorthand for higher genus curves.  (Note also that freeness in our sense implies the related notion of \cite[II.3.1 Definition]{Kollar}.)  We will need the following properties of free curves:

\begin{proposition} \label{prop:freeproperties}
Let $X$ be a projective variety.  Suppose that $s: C \to X$ is an $r$-free curve for some $r \geq 0$.  Then:
\begin{enumerate}
\item For any codimension $\geq 2$ closed subset $Z \subset X$ a general deformation of $s$ will be an $r$-free curve whose image is disjoint from $Z$. (\cite[Lemma 3.8]{LRT24})
\item If $M \subset \Mor(C,X)$ is the irreducible component containing $s$ then the $(\lfloor r \rfloor + 1)$-fold evaluation map $C^{\times (\lfloor r \rfloor + 1)} \times M \to C^{\times (\lfloor r \rfloor + 1)} \times X^{\times (\lfloor r \rfloor + 1)}$ is dominant.   (\cite[Lemma 3.6.(4)]{LRT24})
\end{enumerate}
\end{proposition}

\begin{proof}
The results of \cite{LRT24} are only explicitly stated in characteristic $0$ but the arguments also work in characteristic $p$. 
\end{proof}

The following lemma is frequently useful when working with free curves.

\begin{lemma} \label{lemm:avoidcodim2}
Let $X$ be a $\mathbb{Q}$-factorial normal projective variety.  Suppose that $M \subset \Mor(C,X)$ parametrizes a family of maps $s: C \to X$ with the property that for any codimension $2$ subset $Z \subset X$ there is a deformation of $s$ such that $s(C) \cap Z = \emptyset$.

Let $\mathcal{E}$ be a torsion-free sheaf on $X$.  If for a general map $s: C \to X$ parametrized by $M$ we have $\mu^{min}(s^{*}\mathcal{E}) > 0$, then $\mu^{min}_{s_{*}C}(\mathcal{E}) > 0$.
\end{lemma}

\begin{proof}
Consider a non-trivial torsion-free quotient $\mathcal{E} \to \mathcal{Q}$.  By assumption, a general deformation $s': C \to X$ of $s$ will have the property that $s'(C)$ lies in the locus where $\mathcal{E}$ and $\mathcal{Q}$ are locally free.  Thus the surjection $s'^{*}\mathcal{E} \to s'^{*}\mathcal{Q}$ is also a morphism of locally free sheaves.  By assumption we have $0 < \deg(s'^{*}\mathcal{Q}) = c_{1}(\mathcal{Q}) \cdot s'_{*}C$ so that $\mu_{s_{*}C}(\mathcal{Q}) > 0$.  This implies the desired statement.
\end{proof}

The following weaker positivity condition is also useful.

\begin{definition} \label{defi:almostfree}
Let $X$ be a projective variety over an algebraically closed field and let $r$ be a non-negative number.  We say that a curve $s: C \to X$ is almost $r$-free if there is a smooth open subset $U \subset X$, a smooth morphism $\pi: U \to Z$, and a fiber $F$ of $\pi$ such that $s$ has image in $F$ and is an $r$-free curve in $F$.
\end{definition}

Note that for an almost $r$-free curve there is a positive integer $q$ such that
\begin{equation*}
s^{*}T_{X} \cong \mathcal{O}_{C}^{\oplus q} \oplus \mathcal{E}
\end{equation*}
where $\mathcal{E}$ is a vector bundle satisfying $\mu^{min}(\mathcal{E}) \geq 2g(C) + r$.  By \cite[I.2.17 Theorem]{Kollar} a general fiber $F'$ of $\pi$ will contain a deformation of $s$ which is again an $r$-free curve in $F'$.  Furthermore, if $M$ is an irreducible component of $\Mor(C,X)$ that generically parametrizes almost $r$-free curves then $\dim M = h^{0}(C,s^{*}T_{X})$.

\subsection{Pairs and the MMP} \label{sect:MMP}

In this subsection we work over an algebraically closed field of characteristic $0$.  A pair $(X,\Delta)$ consists of a normal variety $X$ and an effective $\mathbb{Q}$-Weil divisor $\Delta$ such that $K_{X} + \Delta$ is $\mathbb{Q}$-Cartier.  We will use the standard definitions of Kawamata log terminal (klt), divisorially log terminal (dlt) and log canonical (lc) as in \cite{KM98}.  We will frequently use the following results:
\begin{enumerate}
\item If $(X,\Delta)$ is dlt, then for any open neighborhood $U$ of $0 \in N^{1}(X)_{\mathbb{R}}$ there is a divisor $\Delta'$ such that $(X,\Delta')$ is klt and $\Delta' - \Delta \in U$.  (See \cite[Proposition 2.43]{KM98}.)
\item Every dlt pair $(X,\Delta)$ admits a $\mathbb{Q}$-factorialization, i.e.~a small projective birational morphism $\phi: Y \to X$ such that $Y$ is $\mathbb{Q}$-factorial.  Furthermore, the strict transform $\Delta_{Y}$ of $\Delta$ satisfies that $(Y,\Delta_{Y})$ is dlt.  (See \cite[Corollary 1.4.3]{BCHM10}, \cite[Corollary 4.4]{Lohmann13}.) 
\item If $(X,\Delta)$ is a klt Fano pair and $\phi: Y \to X$ is a $\mathbb{Q}$-factorialization, then there is a divisor $\Delta'$ on $Y$ such that $(Y,\Delta')$ is a klt Fano pair.  (See \cite[Lemma 3.1]{GOST15}.)
\end{enumerate}
Combining these results, we see that for every dlt log Fano pair $(X,\Delta)$ there is a small projective birational map $\phi: Y \to X$ and a divisor $\Delta'$ on $Y$ such that $(Y,\Delta')$ is a $\mathbb{Q}$-factorial klt log Fano pair.

\subsection{Positivity of the tangent bundle}

We start by recalling an important theorem of \cite{Ou23}.

\begin{theorem}[{\cite[Theorem 1.4]{Ou23}}] \label{theo:ou}
Let $(X,\Delta)$ be a projective $\mathbb{Q}$-factorial lc pair over an algebraically closed field of characteristic $0$.  Suppose that $-(K_{X} + \Delta)$ is nef.  Then for every surjection $T_{X} \to \mathcal{Q}$ to a torsion-free sheaf $\mathcal{Q}$ we have $c_{1}(\mathcal{Q})$ is pseudo-effective.
\end{theorem}

Equivalently, we have $\mu^{min}_{\alpha}(T_{X}) \geq 0$ for every nef curve class $\alpha$ on $X$.

\begin{proof}
While \cite[Theorem 1.4]{Ou23} only states the result when $\Delta=0$, the argument works in this more general context.  The only point where the proof of \cite[Theorem 1.4]{Ou23} needs to be modified for $\Delta \neq 0$ is when we apply \cite[Theorem 1.10]{Ou23}.  Note however that \cite[Theorem 1.10]{Ou23} is stated for pairs $(X,\Delta)$ and yields the chain of inequalities
\begin{equation*}
K_{\mathcal{F}} \cdot \alpha \geq (K_{X} + \Delta_{ver}) \cdot \alpha \geq K_{X} \cdot \alpha
\end{equation*}
for $\alpha \in \Nef_{1}(X)$ which is all that is needed to finish the proof.
\end{proof}

We will be interested in lc Fano pairs which satisfy a slightly stronger property.

\begin{definition} \label{defi:pqp}
 Let $X$ be a normal projective variety over an algebraically closed field of arbitrary characteristic.  Suppose $\mathcal{E}$ is a torsion-free sheaf on $X$.  We say that $\mathcal{E}$ has the positive quotient property (PQP) if for every surjection $\mathcal{E} \to \mathcal{Q}$ onto a positive rank torsion-free sheaf, the first Chern class $c_{1}(\mathcal{Q})$ lies in $\Eff_{n-1}(X)$ and is not numerically trivial.
\end{definition}

The following result from \cite{CP19} shows that if $T_{X}$ fails the PQP, then there is a rational map associated to the largest quotient of $T_{X}$ with numerically trivial first Chern class.  

\begin{lemma} \label{lemm:trivquotient}
Let $(X,\Delta)$ be a projective $\mathbb{Q}$-factorial lc pair over an algebraically closed field of characteristic $0$.  Suppose that $-(K_{X} + \Delta)$ is nef.  Amongst all surjections $T_{X} \to \mathcal{Q}$ to torsion-free sheaves such that $c_{1}(\mathcal{Q})$ is numerically trivial there is a unique ``maximal quotient'', i.e.~a quotient $T_{X} \to \mathcal{Q}_{max}$ through which all other such morphisms factor.  The kernel $\mathcal{F}$ of this surjection defines an algebraically integrable foliation with rationally connected leaves.
\end{lemma}

\begin{proof}
If there are no quotients of slope $0$, $\mathcal{Q}_{max} = 0$ and the foliation is trivial.  Otherwise, choose a curve class $\alpha$ in the interior of $\Nef_{1}(X)$ and consider the $\alpha$-Harder-Narasimhan filtration of $T_{X}$:
\begin{equation*}
0 = \mathcal{F}_{0} \subset \mathcal{F}_{1} \subset \mathcal{F}_{2} \subset \ldots \subset \mathcal{F}_{s} = T_{X}.
\end{equation*}
Theorem \ref{theo:ou} shows that every successive quotient $\mathcal{F}_{i}/\mathcal{F}_{i-1}$ has non-negative $\alpha$-slope.  Thus every quotient of slope $0$ must be a quotient of the final term $\mathcal{F}_{s}/\mathcal{F}_{s-1}$.   Due to the semistability of $\mathcal{F}_{s}/\mathcal{F}_{s-1}$ we see that it also has slope zero so that $\mathcal{Q}_{max} = \mathcal{F}_{s}/\mathcal{F}_{s-1}$.

\cite[Proposition 1.3.32]{Pang15} shows that $\mathcal{F}_{s-1}$ is a foliation.  Consider a smooth birational model $\phi: X' \to X$ and let $\mathcal{F}'$ denote the foliation on $X'$ induced by $\mathcal{F}_{s-1}$.  Using the compatibility of slopes with pull-back of nef curve classes as established by \cite[Proposition 2.8]{GKP16}, we see that $\mu^{min}_{\phi^{*}\alpha}(\mathcal{F}') > 0$.  Thus $\mathcal{F}'$ (and equivalently $\mathcal{F}_{s-1}$) is algebraically integrable with rationally connected leaves by \cite[Theorem 1.1]{CP19}.
\end{proof}

The arguments of \cite{Ou23} show that the tangent bundle of a $\mathbb{Q}$-factorial dlt Fano variety always has the PQP.

\begin{proposition} \label{prop:kltpqp}
Let $(X,\Delta)$ be a $\mathbb{Q}$-factorial dlt Fano pair over an algebraically closed field of characteristic $0$.  Then for every surjection $T_{X} \to \mathcal{Q}$ to a positive rank torsion-free sheaf $\mathcal{Q}$ we have $c_{1}(\mathcal{Q})$ is pseudo-effective and is not numerically trivial.
\end{proposition}

Equivalently, we have $\mu^{min}_{\alpha}(T_{X}) > 0$ for every nef curve class in the interior of $\Nef_{1}(X)$.

\begin{proof}
By perturbing $\Delta$ we can find an effective divisor $\Delta'$ such that $(X,\Delta')$ is a $\mathbb{Q}$-factorial klt Fano pair.
Theorem \ref{theo:ou} shows that every positive rank torsion-free quotient of $T_{X}$ has pseudo-effective first Chern class.  Suppose for a contradiction that there exists a quotient whose first Chern class is numerically trivial.  Let $\phi: X \dashrightarrow Z$ denote the rational map induced by the maximal quotient $T_{X} \to \mathcal{Q}_{max}$ as in Lemma \ref{lemm:trivquotient}.  Let $D$ denote an effective Weil divisor on $X$ defined by pulling back a very ample divisor $H$ on $Z$ and taking its closure.  Choose $\epsilon > 0$ such that $(X, \Delta' + \epsilon D)$ is klt and $-(K_{X} + \Delta' + \epsilon D)$ is still ample.  Then \cite[Theorem 1.10]{Ou23} shows that $c_{1}(\mathcal{Q}_{max}) - \epsilon D$ is pseudo-effective.  Thus $c_{1}(\mathcal{Q}_{max})$ is not numerically trivial, a contradiction.
\end{proof}

\section{Existence of free curves}

This section is devoted to proving the existence of free curves in the smooth locus of a mildly singular Fano variety.  In the first section, we use characteristic $p$ techniques to demonstrate that a curve $s: C \to X^{sm}$ such that $s^{*}T_{X}$ is ample can be modified to obtain a free curve.  In the later sections, we show to how verify the existence of curves $s: C \to X^{sm}$ such that $s^{*}T_{X}$ is ample and then use these constructions to prove Theorem \ref{theo:dltcase}, Theorem \ref{theo:lccase}, and several variants.  

\subsection{Finding free curves}

We start with a well-known lemma (see e.g.~\cite[Section 2]{Gounelas16}) about Frobenius pullbacks of ample vector bundles on curves.  One can also give an effective version of this lemma using \cite[Corollary 6.2]{Langer04}.

\begin{lemma} \label{lemm:frobpullback}
Let $k$ be a field of characteristic $p$.  Suppose that $C$ is a smooth projective curve over $k$ and $\mathcal{E}$ is an ample vector bundle on $C$.  For any positive number $r$, there is a positive integer $e_{0}$ such that the iterated Frobenius pullback $F_{rel}^{e}: C^{(-e)} \to C$ satisfies $\mu^{min}(F_{rel}^{e*}\mathcal{E}) > r$ for every $e \geq e_{0}$.
\end{lemma}

\begin{proof}
The proof is by induction on the rank.  When $\mathcal{E}$ is a line bundle, then $\mu^{min}(F_{rel}^{e*}\mathcal{E}) = p^{e} \mu(\mathcal{E}) \geq p^{e}$ and it suffices to take $e > \log_{p}(r)$.

In general, first suppose that $F_{rel}^{e*}\mathcal{E}$ is semistable for every $e \geq 0$.  Then $\mu^{min}(F_{rel}^{e*}\mathcal{E}) = p^{e} \mu(\mathcal{E}) \geq p^{e} \frac{1}{\rk(\mathcal{E})}$ so it suffices to choose $e > \log_{p}(r \rk(\mathcal{E}))$.  Next suppose that $F_{rel}^{j*}\mathcal{E}$ is not semistable for some $j \geq 0$.  Since $F_{rel}^{j*}\mathcal{E}$ is ample, every successive quotient in the Harder-Narasimhan filtration is ample and we conclude by induction on the rank.
\end{proof}

The following theorem shows that if we can identify one ``positive'' curve in the smooth locus of $X$ then we can modify it to obtain a free curve in the smooth locus of $X$.  The argument is similar to Mori's proof of the Bend-and-Break theorem: we spread out to characteristic $p$ and use the Frobenius morphism to improve the positivity of the curve.  There are related statements in \cite[Section 4]{KSCT07} and \cite[Theorem 5.5]{Gounelas16}.

\begin{theorem} \label{theo:positivetofree}
Let $X$ be a projective variety of dimension $n$ over an algebraically closed field.  Let $U$ be an open subset of the smooth locus of $X$.  Suppose that $C$ is a smooth projective curve and $s: C \to X$ is a morphism such that $s(C) \subset U$ and $s^{*}T_{X}$ is ample.

Then for any positive number $r$, there is an $r$-free curve $s': C \to X$ whose image is contained in $U$ and whose numerical class in $N^{n-1}(X)_{\mathbb{R}}$ is proportional to $s_{*}C$.
\end{theorem}

\begin{proof}
When $k$ has characteristic $0$, we start by spreading out to characteristic $p$.  Set $Z = X \backslash U$ and choose a resolution of singularities $\phi: Y \to X$ that is an isomorphism over $U$.  Choose a finitely generated extension $K/\mathbb{Q}$ such that $X$, $Z$, $C$, $Y$, and the morphisms $s$, $\phi$ are defined over $K$.  Choose an integral domain $R$ that is finitely generated over $\mathbb{Z}$ whose fraction field is equal to $K$.   We will repeatedly replace $\Spec(R)$ by a smaller open affine subset, which by an abuse of notation we will continue to call $\Spec(R)$.

After perhaps replacing $\Spec(R)$ by an open affine subset,  we can spread $X$, $Z$, and $C$ out to schemes $\mathcal{X}$, $\mathcal{Z}$ which are flat and projective over $\Spec(R)$ and a scheme $\mathcal{C}$ which is smooth projective with geometrically integral fibers over $\Spec(R)$.  We set $\mathcal{U} = \mathcal{X} \backslash \mathcal{Z}$; we may assume $\mathcal{U}$ is smooth over $\Spec(R)$ by shrinking $\Spec(R)$ if necessary.  Consider the relative hom scheme $\Mor_{R}(\mathcal{C},\mathcal{X})$.  After perhaps shrinking $\Spec(R)$ further, the morphism $s: C \to X$ spreads out to a morphism $\tilde{s}: \mathcal{C} \to \mathcal{X}$ whose image is contained in $\mathcal{U}$.  By openness of ampleness, after shrinking further we may ensure that for every $\fp \in \Spec(R)$ the bundle $\tilde{s}_{\fp}^{*}T_{\mathcal{X}_{\fp}}$ is ample. Finally, we can also ensure that $\phi$ spreads out to give a family of resolutions $\phi_{R}: \mathcal{Y} \to \mathcal{X}$ over $\Spec(R)$ that are isomorphisms over $\mathcal{U}$.

Let $\fp$ be a maximal ideal of $R$.  Note that the residue field at $\fp$ is finite and that $\mathcal{C}_{\fp}$ is a smooth projective geometrically integral curve.  By Lemma \ref{lemm:frobpullback}, there is an iterated relative Frobenius morphism $F_{rel}^{e}: \mathcal{C}_{\fp}^{(-e)} \to \mathcal{C}_{\fp}$ such that $\mu^{min}(F_{rel}^{e*}\tilde{s}_{\fp}^{*}T_{\mathcal{X}_{\fp}}) \geq 2g(C) + r$.  Since the residue field is finite, after possibly passing to a higher twist we may ensure that $\mathcal{C}_{\fp}^{(-e)}$ is isomorphic to $\mathcal{C}_{\fp}$ as a $\kappa(\fp)$-scheme.  Thus the composition $\tilde{s}'_{\fp} := \tilde{s}_{\fp} \circ F_{rel}^{e}: \mathcal{C}_{\fp} \to \mathcal{X}_{\fp}$ is an $r$-free curve.  It is clear that the numerical class of $\tilde{s}'_{\fp *}\mathcal{C}_{\fp}$ is proportional to the class of $\tilde{s}_{\fp *}\mathcal{C}_{\fp}$. 

Finally, we deform back to the general fiber using the freeness of $\tilde{s}'_{\fp}$. Since the obstruction group $H^{1}(\mathcal{C}_{\fp}, \tilde{s}_{\fp}'^{*}T_{\mathcal{X}_{\fp}})$ vanishes, \cite[I.2.17 Theorem]{Kollar} shows that every irreducible component $M$ of $\Mor_{R}(\mathcal{C},\mathcal{X})$ that contains $\tilde{s}'_{\fp}$ has dimension at least $H^{0}(\mathcal{C}_{\fp}, \tilde{s}_{\fp}'^{*}T_{\mathcal{X}_{\fp}}) + \dim_{\fp}(\Spec(R))$.  Furthermore, since $\tilde{s}'_{\fp}$ is free its deformations in $\mathcal{X}_{\fp}$ have the expected dimension.  By upper semicontinuity of fiber dimensions, we conclude that the map from $M$ to $\Spec(R)$ is dominant, and in particular, that the image contains the generic point.  Using \cite[Theorem 5]{Nitsure11} we see that a general map parametrized by $M$ over the generic point will be an $r$-free curve $s': C \to X$.  Since a general deformation of $\tilde{s}'_{\fp}$ remains in $\mathcal{U}$, the image of $s'$ is in $U$.

It only remains to verify the proportionality of the numerical classes.  Consider the strict transform over $\phi_{R}: \mathcal{Y} \to \mathcal{X}$ of the constructed family of deformations of $\tilde{s}_{\mathfrak{p}}: \mathcal{C}_{\mathfrak{p}} \to \mathcal{X}_{\mathfrak{p}}$.  Letting $\pi: \mathcal{Y} \to \Spec(R)$ denote the structure map, \cite[Proposition 3.6]{MP12} constructs an injective morphism on numerical spaces of Cartier divisors from the geometric fiber of $\pi$ over the generic point of $\Spec(R)$ to the geometric fiber of $\pi$ over $\fp$. Thus on $Y$ we conclude that the numerical class of $s'_{*}C$ is proportional to the class of $s_{*}C$ as an element of $N_{1}(Y)_{\mathbb{R}}$.  Using the isomorphism $N_{1}(Y)_{\mathbb{R}} \cong N^{n-1}(Y)_{\mathbb{R}}$ and the injectivity of $\phi^{*}: N^{n-1}(X)_{\mathbb{R}} \to N^{n-1}(Y)_{\mathbb{R}}$ we obtain the desired proportionality on $X$.

When $k$ has characteristic $p$, we run a similar argument.  Choose a finitely generated extension $K/\mathbb{F}_{p}$ such that $X$, $C$, and the morphism $s$ are defined over $K$.  We may also ensure that there is a finite collection of Weil divisors $\{ D_{i} \}$ on $X$ that span $N_{n-1}(X)_{\mathbb{R}}$ and are defined over $K$.  Choose an integral domain $R$ that is finitely generated over $\mathbb{F}_{p}$ whose fraction field is equal to $K$.  By the same shrinking process as before, we may assume the same properties as before for the families $\mathcal{X}, \mathcal{C}$ and the morphism $\tilde{s}: \mathcal{C} \to \mathcal{X}$ obtained by spreading out over $\Spec(R)$.  After shrinking $\Spec(R)$ further, we may also suppose that every $D_{i}$ spreads out to a divisor $\mathcal{D}_{i}$ that is flat over $\Spec(R)$.

Let $\fp$ be a maximal ideal of $\Spec(R)$ so that the residue field at $\fp$ is finite.  By precomposing $\tilde{s}_{\fp}$ with an $e$th iterate of the Frobenius for an appropriate choice of $e$, we find an $r$-free curve over $\fp$, and this deforms to a map $\tilde{s}': \mathcal{C} \to \mathcal{X}$ such that the restriction to the generic fiber is an $r$-free curve $s': C \to X$.

It only remains to prove proportionality of numerical classes.  We compare the map $\tilde{s}': \mathcal{C} \to \mathcal{X}$ to the map $\hat{s}: \hat{\mathcal{C}} \to \mathcal{X}$ obtained by precomposing $\tilde{s}$ by the $e$th iterate of the relative Frobenius for $\mathcal{C} \to \Spec(R)$.  By taking the difference of the image cycles we obtain a family of $1$-cycles in the fibers of $\mathcal{X} \to \Spec(R)$ such that the $1$-cycle over $\fp$ is trivial.  In particular, on $\mathcal{X}_{\fp}$ the intersection (in the sense of Remark \ref{rema:defineint}) of this $1$-cycle against any of the $\mathcal{D}_{i,\mathfrak{p}}$  is zero.  By conservation of number (\cite[Theorem 10.2]{Fulton84}) we see that the same is true in the generic fiber, showing that $s'_{*}C$ is numerically proportional to a multiple of $s_{*}C$.
\end{proof}

\subsection{Finding starting curves}

In order to apply Theorem \ref{theo:positivetofree}, we must find ``positive'' curves in the smooth locus of Fano varieties.

\begin{proposition} \label{prop:lccurveexists}
Let $X$ be a normal projective variety of dimension $n$ over an algebraically closed field $k$.  Suppose there is an ample Cartier divisor $H$ such that $\mu^{min}_{H^{n-1}}(T_{X}) > 0$.  Then there is a curve $s:C \to X$ whose image is contained in the smooth locus of $X$ such that $s^{*}T_{X}$ is ample.
\end{proposition}

The argument looks a little different in characteristic $0$ and in characteristic $p$.  

\begin{proof}
In characteristic $0$, by repeatedly applying \cite[1.2 Theorem]{Flenner84} we see that a general complete intersection curve $C$ constructed by intersecting divisors linearly equivalent to a sufficiently positive multiple of $H$ will satisfy $\mu^{min}(s^{*}T_{X}) > 0$.  Since we are in characteristic $0$ this implies ampleness of $s^{*}T_{X}$ by \cite[Theorem 2.4]{Hartshorne71}.

In characteristic $p$, we denote the $H^{n-1}$-HN-filtration of $T_{X}$ by
\begin{equation*}
0 = \mathcal{F}_{0} \subset \mathcal{F}_{1} \subset \ldots \subset \mathcal{F}_{s} = T_{X}.
\end{equation*}
For a tuple of positive integers $(m_{1},\ldots,m_{n-1})$ such that each divisor $m_{i}H$ is very ample we let $\mathcal{C}_{\vec{m}}$ denote the generic fiber of the corresponding family of complete intersection curves.   Theorem \ref{theo:normalsemistability} shows that for an appropriate choice of $\vec{m}$ the restriction of each successive quotient $\mathcal{F}_{i}/\mathcal{F}_{i-1}$ to $\mathcal{C}_{\vec{m}}$ is strongly semistable of positive slope.  Since $\mathcal{C}_{\vec{m}}$ is smooth and geometrically integral, \cite[Proposition 3]{Langton75} shows that strong semistability is preserved by base change to the algebraic closure and we conclude that $T_{X}|_{\mathcal{C}_{\vec{m}}}$ is ample by \cite[Theorem 2.1]{Barton71}.  Spreading out, we see that for a general complete intersection curve $s:C \to X$ the sheaf $s^{*}T_{X}$ is ample.
\end{proof}

\subsection{Proving main results} \label{sect:mainexistence}

We are now prepared to prove some of the results from the introduction.

\begin{theorem} \label{theo:lcandantinef}
Let $(X,\Delta)$ be a $\mathbb{Q}$-factorial projective lc pair over an algebraically closed field of characteristic $0$.  Suppose $-(K_{X} + \Delta)$ is nef.  Then the following conditions are equivalent:
\begin{enumerate}
\item There is no positive rank quotient of $T_{X}$ with numerically trivial first Chern class. 
\item $X$ admits a $1$-free curve.
\item For every $m > 0$ we have $h^{0}(X^{sm}, \Omega_{X^{sm}}^{\otimes m}) = 0$.
\end{enumerate}
If these conditions hold, then there is a smooth projective curve $C$ such that for every $r \geq 0$ there is an $r$-free morphism $s: C \to X$.
\end{theorem}

\begin{proof}
(1) $\implies$ (2): suppose that $(X,\Delta)$ is a $\mathbb{Q}$-factorial lc pair with $-(K_{X} + \Delta)$ nef such that $T_{X}$ has no positive rank quotient with numerically trivial first Chern class.  By Theorem \ref{theo:ou} $T_{X}$ satisfies the PQP.  Then Proposition \ref{prop:lccurveexists} and Theorem \ref{theo:positivetofree} combine to prove the existence of $r$-free curves for every $r \geq 0$.

(2) $\implies$ (3):  suppose that $X$ carries a $1$-free curve.  Then through any general point $x \in X^{sm}$ there is a curve $s: C \to X^{sm}$ such that $s^{*}T_{X^{sm}}$ is ample.  Thus every section of the locally free sheaf $\Omega_{X^{sm}}^{\otimes m}$ must vanish at $x$.  Since this is true for a dense set of points $x \in X^{sm}$, we conclude that the space of global sections is $0$.

(3) $\implies$ (1): Let $\phi: Y \to X$ be a resolution of singularities.  Then we have $h^{0}(Y, \Omega_{Y}^{\otimes m}) = 0$ for every $m > 0$ as well.  In particular $\Pic^{0}_{Y/\mathbb{C}} = 0$, so every numerically trivial Cartier divisor on $Y$ is torsion.  By appealing to $\mathbb{Q}$-factoriality and pulling back divisors from $X$ to $Y$, we see the same is true of numerically trivial Weil divisors on $X$.

Suppose we had a surjection $T_{X} \to \mathcal{Q}$ where $c_{1}(\mathcal{Q})$ was numerically trivial.  Since a torsion sheaf always has pseudo-effective first Chern class, we may suppose $\mathcal{Q}$ is torsion-free.  Then for any $m > 0$ we have a chain of surjections on the smooth locus $X^{sm}$
\begin{equation*}
T_{X^{sm}}^{\otimes m \rk(\mathcal{Q})} \to \mathcal{Q}|_{X^{sm}}^{\otimes m \rk(\mathcal{Q})} \to \left(\bigwedge\nolimits^{\rk(\mathcal{Q})} \mathcal{Q}|_{X^{sm}} \right)^{\otimes m}.
\end{equation*}
In particular, when $m$ satisfies $mc_{1}(\mathcal{Q}) \sim 0$ we obtain a surjection $T_{X^{sm}}^{\otimes m\rk(\mathcal{Q})} \to \mathcal{O}_{X^{sm}}$.  Dually, we get a non-vanishing tensor field on $X^{sm}$, a contradiction.
\end{proof}

\begin{proof}[Proof of Theorem \ref{theo:lccase}:]
Follows immediately from Theorem \ref{theo:lcandantinef}.
\end{proof}

\begin{proof}[Proof of Theorem \ref{theo:maintheoremcharp}:]
Follows immediately from Theorem \ref{theo:positivetofree} and Proposition \ref{prop:lccurveexists}.
\end{proof}

\begin{proof}[Proof of Theorem \ref{theo:dltcase}:]  As discussed in Section \ref{sect:MMP}, there is a $\mathbb{Q}$-factorialization $\phi: X' \to X$ and a divisor $\Delta'$ on $X'$ such that $(X',\Delta')$ is $\mathbb{Q}$-factorial klt log Fano. Theorem \ref{theo:lcandantinef} and Proposition \ref{prop:kltpqp} show that there is a curve $C$ such that for every non-negative $r$ there is an $r$-free morphism $s': C \to X'$.   Since a general deformation of an $r$-free curve $s': C \to X'$ will miss the codimension $\geq 2$ sublocus of $X'$ where $\phi$ is not an isomorphism, by taking images on $X$ we obtain the desired curves.
\end{proof}

\subsection{Representing nef classes by free curves} \label{sect:variants}

In this section we prove two variants of Theorem \ref{theo:dltcase} which give more details about which numerical classes are represented by free curves.  We first discuss the situation for pairs which are not necessarily Fano.

\begin{proposition} \label{prop:almostfree}
Let $(X,\Delta)$ be a dlt pair over an algebraically closed field of characteristic $0$.  Suppose that $\mathcal{F}$ is a $(K_{X} + \Delta)$-negative extremal face of the cone $\Nef_{1}(X) + \Eff_{1}(X)_{K_{X} + \Delta \geq 0}$.  Then there is some curve $C$ such that for any $r \geq 0$ there is an almost $r$-free curve $s: C \to X^{sm}$ whose class is in $\mathcal{F}$.
\end{proposition}

In particular, if $(X,\Delta)$ is a dlt pair such that $K_{X} + \Delta$ fails to be pseudo-effective then there is an almost $r$-free curve $s: C \to X^{sm}$. 

\begin{proof}
Let $\phi: X' \to X$ be a $\mathbb{Q}$-factorialization of $(X,\Delta)$ and let $\Delta'$ be the strict transform of $\Delta$.  There is an extremal face $\mathcal{F}'$ of the analogous cone for $(X',\Delta')$ such that $\phi_{*}\mathcal{F}' = \mathcal{F}$.  Since a general deformation of an almost free curve will avoid the codimension $2$ locus where $\phi$ is not an isomorphism, the desired statement for $X$ follows from the statement for $X'$.  Thus we may assume $(X,\Delta)$ is $\mathbb{Q}$-factorial.  After perturbing $\Delta$ we may also suppose that $(X,\Delta)$ is klt.

Let $D$ be a pseudo-effective $\mathbb{Q}$-divisor such that the nef curve classes with vanishing intersection against $D$ are exactly the classes in the face $\mathcal{F}$.  \cite[Lemma 4.3]{Lehmann12} shows that (after rescaling) we may write $D = K_{X} + \Delta + A$ for some ample $\mathbb{Q}$-divisor $A$ such that $(X,\Delta + A)$ is klt.  By \cite[Theorem C]{BCHM10} we can run a $D$-MMP with scaling of $A$ to obtain a birational contraction $\psi: X \dashrightarrow X'$ such that $\psi_{*}D$ is nef.  We can write $\psi_{*}D \sim_{\mathbb{Q}} K_{X'} + \psi_{*}\Delta + (1-\epsilon)\psi_{*}A + \epsilon E + H$ for some choice of $\epsilon > 0$, effective $\mathbb{Q}$-divisor $E$ on $X'$, and ample $\mathbb{Q}$-divisor $H$ on $X'$ such that $(X', \psi_{*}\Delta + (1-\epsilon)\psi_{*}A + \epsilon E)$ is klt.  Thus we may apply the basepoint free theorem (see \cite[Theorem 3.3]{KM98}) to find a fiber space $\pi: X' \to Z$ which contracts precisely the curves with vanishing intersection against $\psi_{*}D$.  Note that for a general fiber $F$ of $\pi$ the pair $(F,(\psi_{*}\Delta + (1-\epsilon)\psi_{*}A + \epsilon E)|_{F})$ will be klt Fano.  In particular, by Theorem \ref{theo:dltcase} there is a curve $C$ such that for any $r \geq 0$ we have an $r$-free curve $s: C \to F$.  By freeness, a general deformation of $s$ will avoid the indeterminacy locus of $\psi^{-1}$.  Thus the strict transform of these maps on $X$ will be almost free curves.  Since they have vanishing intersection against $D$, their numerical classes lie in $\mathcal{F}$.
\end{proof}

Our next result addresses the homological classes of free curves.  Let $X$ be a projective variety and consider the set of cycles on $X$ with a fixed codimension $>1$.  It is frequently a delicate question whether or not a given homological class is represented by an effective cycle, and even harder to determine whether it is represented by a ``positive'' cycle.  We prove a strong result in this direction for curves on dlt Fano pairs:

\begin{theorem} \label{theo:kltallclasses}
Let $(X,\Delta)$ be a dlt Fano pair over an algebraically closed field of characteristic $0$.  Every rational ray in the interior of $\Nef^{n-1}(X)$ is represented by a $1$-free curve.
\end{theorem}

Before proving Theorem \ref{theo:kltallclasses}, we need a more precise version of Proposition \ref{prop:lccurveexists} for dlt varieties. 

\begin{proposition} \label{prop:positivedense}
Let $(X,\Delta)$ be a $\mathbb{Q}$-factorial dlt log Fano pair over an algebraically closed field of characteristic $0$.  There is a dense set of rays in $\Nef_{1}(X)$ which are represented by curves $s: C \to X$ whose images are contained in the smooth locus of $X$ and which satisfy $\mu^{min}(s^{*}T_{X}) > 0$.
\end{proposition}

\begin{proof}
\cite[Corollary 1.3.1]{BCHM10} proves that $X$ is a Mori Dream Space.  Suppose that $\phi: X \dashrightarrow X'$ is a small $\mathbb{Q}$-factorial modification.  Let $A$ be a very ample divisor on $X'$.  \cite[1.2 Theorem]{Flenner84} implies that for some sufficiently large $m$ we can find a curve $C'$ which is a general complete intersection of elements in $|mA|$ and satisfies
\begin{equation*}
\mu^{min}(T_{X'}|_{C'}) = \mu^{min}_{C'}(T_{X'}).
\end{equation*}
Since $C'$ is a general complete intersection, we may ensure that it is contained in the locus where $\phi^{-1}$ is an isomorphism.  Let $C$ denote the strict transform curve on $X$.  By \cite[Corollary A.22]{GKP14} we see that $\mu^{min}_{C'}(T_{X'}) = \mu^{min}_{C}(T_{X})$.  Combining the equalities above with Proposition \ref{prop:kltpqp}, we have
\begin{equation*}
\mu^{min}(T_{X}|_{C}) = \mu^{min}(T_{X'}|_{C'}) = \mu^{min}_{C'}(T_{X'}) = \mu^{min}_{C}(T_{X}) > 0
\end{equation*}
showing that $C$ has the desired properties. 

It only remains to show that the rays spanned by classes of the above form are dense in $\Nef_{1}(X)$.  Given any small $\mathbb{Q}$-factorial modification $\phi: X \dashrightarrow X'$, let $N_{\phi}$ denote the strict transform of the classes of the form $A^{n-1}$ on $X'$ where $A$ is an ample $\mathbb{Q}$-Cartier divisor.  \cite[Example 5.4]{LX19} shows that $\Nef_{1}(X) = \cup_{\phi} \overline{N_{\phi}}$ and this concludes the proof.
\end{proof}

We now can apply this to representing rays of the nef cone by free curves.
\begin{proof}[Proof of Theorem \ref{theo:kltallclasses}:]
As discussed in Section \ref{sect:MMP}, there is a $\mathbb{Q}$-factorialization $\phi: X' \to X$ and a divisor $\Delta'$ on $X'$ such that $(X',\Delta')$ is $\mathbb{Q}$-factorial klt log Fano.  By Lemma \ref{lemm:nonqfactorialcone} the induced injective map $\phi^{*}: \Nef^{n-1}(X) \to \Nef^{n-1}(X')$  preserves the property of being a boundary or interior class.  Thus it suffices to prove the statement when $(X,\Delta)$ is a $\mathbb{Q}$-factorial klt log Fano pair.  

By combining Proposition \ref{prop:positivedense} and Theorem \ref{theo:positivetofree}, we see there is a dense set of rays $\{ R_{i} \}$ in the interior of $\Nef_{1}(X)$ which are represented by classes of $1$-free curves.  For any rational ray $R$ in the interior of the nef cone, we can choose a simplicial cone in the interior of $\Nef_{1}(X)$ containing it.  By choosing sufficiently small open neighborhoods of the generators of this simplicial cone and appealing to density, we can find a finite set of $1$-free curves $\{ C_{j} \}_{j=1}^{r}$ such that the cone generated by their numerical classes contains $R$ in its interior.  In particular, there is a sum of the $C_{j}$ with positive integer coefficients which has numerical class lying on $R$.  
By choosing appropriate deformations of the $C_{j}$, we can glue them to obtain a stable map $s': C' \to X$ such that the restriction of $s'$ to each component is $1$-free.   Since $H^{1}(C',s'^{*}T_{X}) = 0$, this curve can be smoothed to obtain a $1$-free curve.
\end{proof}

\begin{remark}
With a little additional effort along the lines of the proof of Proposition \ref{prop:almostfree}, one can show that every rational ray on the boundary of the nef cone is represented by an almost free curve.
\end{remark}

\section{Rational curves on terminal Fano threefolds}

In this section we prove the existence of free rational curves in the smooth locus of a terminal Fano threefold.  Throughout we work over an algebraically closed field of characteristic $0$.
Starting with a family of free (higher genus) curves in the smooth locus of $X$, we carefully study how such curves can break.  This allows us to construct a dominant family of rational curves which meet the singular locus no more than twice.  Finally, an application of \cite[Corollary 5.9]{KM99} yields a free rational curve.

\begin{remark}
When $X$ is a terminal Fano variety with LCI singularities, there is a well-known direct proof of the existence of free rational curves in the smooth locus.  Suppose $M \subset \Mor(\mathbb{P}^{1},X)$ is an irreducible component parametrizing a dominant family of rational curves $s: \mathbb{P}^{1} \to X$.  Let $\phi: Y \to X$ be a resolution of singularities and let $s': \mathbb{P}^{1} \to Y$ denote the strict transform of a general $s$.  By \cite[II.1.3 Theorem]{Kollar}
\begin{equation*}
-K_{X} \cdot s_{*}\mathbb{P}^{1} + \dim X \leq \dim M = -K_{Y} \cdot s'_{*}\mathbb{P}^{1} + \dim Y.
\end{equation*}
Thus $K_{Y/X} \cdot C' \leq 0$.  Since $K_{Y/X}$ is effective and $C'$ is general in a dominant family, we conclude that equality holds, i.e.~$\phi(C')$ avoids the singular locus of $X$.
\end{remark}

We start with a couple helpful lemmas.  The first is a basic observation about families of curves which are generically disjoint from an effective divisor $E$.

\begin{lemma} \label{lemm:localint}
Let $X'$ be a smooth projective variety and let $E$ be an effective divisor on $X'$.  Suppose $B $  is a one-parameter family of stable maps in $\overline{\mathcal{M}}_{g,n}(X',\beta)$ whose general member is disjoint from $E$.

Let $s: C \to X'$ be a morphism parametrized by $B$ and suppose that there is an irreducible component $T \subset C$ and a point $q \in T$ such that $s(q) \in \Supp(E)$ but $s(T) \not \subset \Supp(E)$.  Then $C$ must have an irreducible component $R$ that meets $T$ at $q$ such that $s(R) \subset \Supp(E)$.
\end{lemma}

\begin{proof}
After perhaps replacing $B$ by a cover, we let $\varphi: \mathcal{C} \to B$ denote the universal family equipped with the evaluation map $\mathfrak{s}: \mathcal{C} \to X'$.  Note that every irreducible component of the divisor $\mathfrak{s}^{*}E$ is $\varphi$-vertical.  Since $T$ meets $\mathfrak{s}^{*}E$ at $q$ but is not contained in it, there must be an irreducible component $R$ of $\mathfrak{s}^{*}E$ meeting $T$ at $q$.  Since every irreducible component of $\mathfrak{s}^{*}E$ is $\varphi$-vertical, $R$ must be an irreducible component of $C$.
\end{proof}

The next lemma will allow us to compare intersection numbers on our terminal threefold $X$ and on a resolution of singularities $\phi: X' \to X$.  In practice, we will apply it when $T$ is a free rational curve on $X'$.

\begin{lemma} \label{lem:negativeKX'degree}
Let $\phi : X' \to X$ be a resolution of singularities of a Fano variety with canonical singularities.  Suppose $B$ is a one-parameter family of stable maps in $\overline{\mathcal{M}}_{g,n}(X',\beta)$ whose general member is disjoint from the $\phi$-exceptional locus.  Let $s : C_0 \to X'$ be any stable map parametrized by $B$.  Suppose $S \subset C_0$ is a connected subcurve such that
\begin{enumerate}
\item $s$ does not contract $S$ to a point in $X'$, and
\item for any irreducible component $T \subset \overline{C_0\setminus S}$ meeting $S$, $K_{X'} \cdot s_*T \leq 0$ and $s(T)$ is not contained in the exceptional locus of $\phi$.  
\end{enumerate}
Then $(\phi^{*}K_{X} - K_{X'}) \cdot  s_{*}S \geq 0$.  Moreover, $-K_{X'} \cdot s_{*}S \geq 1$ if $\phi \circ s$ does not contract $S$ or if $s(S \cap \overline{C_0 \setminus S})$ meets a $\phi$-exceptional divisor which has positive discrepancy.
\end{lemma}

\begin{proof} 
We denote the $\phi$-exceptional divisors by $\{E_i\}_{i=1}^{q}$.  Let $C_1 = \overline{C_0 \setminus S}$.  After perhaps replacing $B$ by a cover, there is a family of curves $\varphi : \mathcal{C} \rightarrow B$ with a morphism $\mathfrak{s} : \mathcal{C} \to X'$ realizing our family of stable maps where $\mathcal{C}$ is a surface with $\mathbb{Q}$-factorial singularities and every fiber of $\varphi$ is reduced. 

We claim $E_{i} \cdot s_*S \leq 0$ for each $i$.  
Indeed, since the general curve is disjoint from $E_{i}$ the pullback $\mathfrak{s}^{*}E_i$ is $\varphi$-vertical.  Suppose $\mathfrak{s}^{*}E_i|_{C_0}$ contains an irreducible component $R$ of $C_{0}$.  Then $R \cdot C_0 = 0$ because $C_0$ is the class of a $\varphi$-fiber.  In particular, if $\overline{R} \subset C_0$ is the union of $R$ with all components in $C_0$ adjacent to it, then $R \cdot \overline{R} = 0$.  If $S$ meets an irreducible component $R \subset \mathfrak{s}^{*}E_i|_{C_0}$, then $R \subset S$ by assumption (2).  Hence, $R \cdot S \leq R \cdot \overline{R} = 0$, with equality if and only if $S$ contains $\overline{R}$.   As this holds for all $R \subset \mathfrak{s}^{*}E_i|_{C_0}$, we conclude $\mathfrak{s}^*E_i \cdot S \leq 0$.  Since $X$ has canonical singularities, the inequalities $E_{i}  \cdot s_{*}S \leq 0$ for all $i$ together imply the first statement.

To see the last statement, first note that since $K_{X'}$ and $s_{*}S$ are integral classes their intersection must be an integer.  
Thus it suffices to show that either $K_X \cdot (\phi \circ s)_* S < 0$ or $E_{i} \cdot s_*S < 0$ for some exceptional divisor $E_{i}$ on $X'$ with positive discrepancy.    
By assumption (1), the image $s(S)$ has dimension one in $X'$.  Because $X$ is Fano, if $\phi \circ s$ does not contract $S$, $K_X \cdot (\phi \circ s)_* S < 0$.  So we may assume each component of $s(S)$ is contracted by $\phi$.  
By hypothesis, there is an exceptional divisor $E_i$ with positive discrepancy and a point $p \in S \cap C_1$ such that $s(p) \in E_i$.  By Lemma \ref{lemm:localint}, there must be a curve $R \subset S \cap \mathfrak{s}^{-1}E_i$ that contains $p$.  For this component, $\overline{R} \not\subset S$, proving $\mathfrak{s}^*E_i \cdot S \leq R  \cdot S < 0$.
\end{proof}

Given a projective variety $X$ and a rational curve $s: \mathbb{P}^{1} \to X$, we say that $s$ meets the singular locus $k$ times if the complement of $s^{-1}(X^{sm})$ is a union of $k$ closed points of $\mathbb{P}^{1}$.

\begin{theorem} \label{theo:terminaldimn}
Let X be a terminal Fano variety.  There is a dominant family of rational curves on $X$ such that the general curve meets the singular locus of $X$ at most $\dim X - 1$ times.
\end{theorem}

\begin{proof}
Set $n = \dim X$.  We may suppose $n \geq 3$, as otherwise $X$ is smooth.  Suppose for a contradiction that every rational curve that deforms in a dominant family meets the singular locus at least $n$ times.

Let $\phi: X' \to X$ be a resolution of singularities of $X$ which is an isomorphism away from the singular locus of $X$.  Let $\{ a_i \}_{i=1}^{t}$ be the discrepancies of the $\phi$-exceptional divisors $\{ E_{i}\}_{i=1}^{t}$ on $X'$ and set $a_{\min} = \min \{ a_1, \ldots, a_{t}, 1\}$.

By Theorem \ref{theo:dltcase}, there is a curve $C$ of genus $g$ such that for any $r \geq 0$ the variety $X$ admits an $r$-free curve $s: C \to X$.  By possibly replacing $C$ by a finite cover we may suppose that $g \geq 2$.  Choose a positive integer $r$ so that $a_{\min}r > ng+1$.  

According to the previous paragraph, we may select an irreducible component $M$ of $\Mor(C,X)$ that generically parametrizes $2r$-free curves. Let $M'$ be the irreducible component of $\Mor(C,X')$ parametrizing the strict transforms $s': C \to X'$ of general curves parametrized by $M$. Let $Z \subset X'$ denote the union of the $\phi$-exceptional divisors and the locus in $X'$ swept out by all families of non-free rational curves whose degree against a fixed ample divisor is at most the degree of $s'_{*}C$. 

We let $d$ denote the $-K_{X}$-degree of the curves parametrized by $M$.  Note that $d$ is also the $-K_{X'}$-degree of the curves parametrized by $M'$.  Since $M$ has the expected dimension  we have 
\begin{equation*}
\dim M = d + (1-g)n.
\end{equation*}

The image of a general curve $s': C \to X'$ parametrized by $M'$ will not be contained in $Z$.  By applying \cite[Proposition 2.2]{JLR25} to $M'$, we see that for $k = \lfloor  \frac{d-n(g-1)}{n+1} \rfloor$ the closure of the image of $M'$ in $\overline{\mathcal{M}}_{g,0}(X)$ parametrizes a stable map $s': C' \to X'$  where $C'$ contains:
\begin{itemize}
\item an irreducible component $C_{0}$ isomorphic to $C$, and
\item $k$ trees of rational curves $T_{i}$ whose image in $X'$ is not contained in $Z$.
\end{itemize}
Because $C$ is $2r$-free, $d \geq (2r + 2g)n$ and thus $k$ is at least $r$.  From the definition of $k$ we also see that $(n+1)k > d-n(g-1)-(n+1) = d - ng-1$.

For each $i$, we can choose a component $R_i \subset T_i$ whose image in $X'$ is not contained in $Z$.  By construction $R_{i}$ deforms in a dominant family on $X'$ and thus  $-K_{X'} \cdot R_{i} \geq 2$.  By assumption, each $R_i$ meets the $\phi$-exceptional locus in at least $n$ points. By Lemma \ref{lemm:localint} there must be at least $n-1$ connected components of $T_i \setminus R_i$ not adjacent to $C_{0}$. By Lemma \ref{lem:negativeKX'degree}, the $-K_{X'}$-degree of each of these connected components must be at least $1$. Moreover, either there are $n$ such connected components, or $R_i$ meets the $\phi$-exceptional locus along a point which is not incident to any of these components. Altogether we see that the $-K_X$-degree of $\phi_* s_* T_i$ is at least $2 + (n-1) + a_{\min}$. It follows that $d$ is at least
\begin{align*}
(n+1+a_{\min})k & = (n+1)k + a_{\min}k \\
& > d-ng-1 + a_{\min}r \\
& > d-ng-1 + ng+1 = d
\end{align*}
This is impossible.
\end{proof}

\begin{proof}[Proof of Theorem \ref{theo:terminalcase}:]
When $X$ is a terminal Fano threefold, Theorem \ref{theo:terminaldimn} shows that there is a dominant family of rational curves on $X$ which meet the singular locus at most $2$ times.  Recall that every terminal Fano threefold has LCIQ singularities (see \cite[Corollary 5.39]{KM98}).  Thus the existence of a free rational curve follows from \cite[5.9 Corollary]{KM99} which shows that any terminal Fano variety that has LCIQ singularities and admits a dominant family of rational curves meeting the singular locus at most twice will also have a free rational curve in its smooth locus.
\end{proof}

\section{Log free curves} \label{sect:logfree}

In this section we extend previous results to the log setting.  Given a $\mathbb{Q}$-Weil divisor $\Delta$, we define the Weil divisor $\lfloor \Delta \rfloor$ by rounding down the coefficients of the irreducible divisors in the formal sum defining $\Delta$.

\subsection{Log tangent bundle and positivity}

Let $(X,\Delta)$ be a projective lc pair. 
There is a smooth open subset $V \subset X$ such that $X \backslash V$ has codimension $\geq 2$ and $\lfloor \Delta \rfloor|_{V}$ is a SNC divisor.  We call the largest such open set $V$ the SNC locus of $(X,\Delta)$.  The log tangent bundle $T_{V}(- \log \lfloor \Delta \rfloor)$ is the kernel of the map $T_{V} \to \oplus_{i} N_{\Delta_{i} \cap V}$, where the $\Delta_{i}$ are the irreducible components of $\lfloor \Delta \rfloor$.

\begin{definition} \label{defi:logtangentbundle}
Let $(X,\Delta)$ be a projective lc pair with SNC locus $V$.  We define the log tangent bundle $T_{(X,\Delta)}^{log}$ to be the unique saturated subsheaf of $T_{X}$ that extends $T_{V}(-\log \lfloor \Delta \rfloor) \subset T_{X}|_{V}$.  
\end{definition}

The log tangent bundle satisfies a natural positivity property.

\begin{theorem}[\cite{Ou23}] \label{theo:logtangentwqp}
Let $(X,\Delta)$ be a $\mathbb{Q}$-factorial projective lc pair over an algebraically closed field of characteristic $0$.  Suppose that $-(K_{X} + \Delta)$ is nef.  For every torsion-free quotient $T_{(X,\Delta)}^{log} \to \mathcal{Q}$ the divisor class $c_{1}(\mathcal{Q})$ is pseudo-effective.
\end{theorem}

\cite{Ou23} proves this statement for the orbifold tangent bundle as defined by \cite{CP19}.  One could prove Theorem \ref{theo:logtangentwqp} by comparing the orbifold tangent bundle and $T^{log}_{(X,\Delta)}$ using \cite[D\'efinition 2.10]{Claudon17}.  We will instead prove the desired positivity simply by repeating the arguments of \cite{Ou23}. 

\begin{proof}
Assume for a contradiction that the statement fails.  Then there is a nef curve class $\alpha$ such that $\mu^{min}_{\alpha}(T^{log}_{(X,\Delta)}) < 0$.  On the other hand, since $-K_{X} - \Delta$ is nef the divisor $-K_{X} - \lfloor \Delta \rfloor$ is pseudo-effective and thus $\mu_{\alpha}(T^{log}_{(X,\Delta)}) \geq 0$.   We conclude that the $\alpha$-Harder-Narasimhan filtration of $T^{log}_{(X,\Delta)}$ includes a term with positive slope.  By taking the last such term, we find $\mathcal{G} \subset T^{log}_{(X,\Delta)}$ such that $\mu^{min}_{\alpha}(\mathcal{G}) > 0$, $\mu^{max}_{\alpha}(T^{log}_{(X,\Delta)}/\mathcal{G}) \leq 0$, and $\mu_{\alpha}(T^{log}_{(X,\Delta)}/\mathcal{G}) < 0$.

Let $\phi: Y \to X$ be a log resolution of $(X,\Delta)$ which is an isomorphism over the SNC locus and let $\Delta_{Y}$ denote the sum of the strict transform of $\lfloor \Delta \rfloor$ with the reduced $\phi$-exceptional divisor.  Then $(Y,\Delta_{Y})$ is a smooth SNC pair.  By \cite[Corollary A.22]{GKP14} we see that the terms in the $\phi^{*}\alpha$-Harder-Narasimhan filtration of $T_{Y}(-\log \Delta_{Y})$ will coincide with the terms of the $\alpha$-Harder-Narasimhan filtration of $T^{log}_{(X,\Delta)}$ when restricted to the locus where $\phi$ is an isomorphism.  In particular, there is a subsheaf $\mathcal{G}_{Y} \subset T_{Y}(-\log \Delta_{Y})$ that is isomorphic to $\mathcal{G}$ over the SNC locus of $X$ and satisfies the same slope inequalities as $\mathcal{G}$ on $X$.  By \cite[Theorem 1.4]{CP19} the saturation $\mathcal{F}_{Y}$ of $\mathcal{G}_{Y}$ in $T_{Y}$ defines an algebraically integrable foliation corresponding to a rational map $\pi_{Y}: Y \dashrightarrow Z$.  By \cite[Proposition 2.17]{Claudon17} we have that $c_{1}(\mathcal{G}_{Y}) = c_{1}(\mathcal{F}_{Y}) - \Delta_{Y,hor}$ where $\Delta_{Y,hor}$ denotes the horizontal part of $\Delta_{Y}$ for the rational map $\pi_{Y}$.  

Passing back down to $X$, we see that the saturation of $\mathcal{G}$ inside $T_X$, which we denote by $\mathcal{F}$, is the foliation induced by a rational map $\pi: X \dashrightarrow Z$.  We define $\Delta_{hor}$ and $\Delta_{ver}$ to be the horizontal and vertical parts of $\Delta$ with respect to $\pi$.  Furthermore we know that $c_{1}(\mathcal{G}) = c_{1}(\mathcal{F}) - \lfloor \Delta_{hor} \rfloor$.
On the other hand, \cite[Theorem 1.10]{Ou23} shows that $-c_{1}(\mathcal{F}) - K_{X} - \Delta_{ver}$ is pseudo-effective.  Thus the same is true of $-c_{1}(\mathcal{F}) - K_{X} - \lfloor \Delta_{ver} \rfloor$.  Since we can identify
\begin{align*}
c_{1}(T^{log}_{(X,\Delta)}/\mathcal{G}) = -K_{X} - \lfloor \Delta \rfloor - \left( c_{1}(\mathcal{F}) - \lfloor \Delta_{hor} \rfloor \right)
& = -c_{1}(\mathcal{F}) - K_{X} - \lfloor \Delta_{ver} \rfloor
\end{align*}
this contradicts the fact that $\mu_{\alpha}(T^{log}_{(X,\Delta)}/\mathcal{G}) < 0$.
\end{proof}

Just as with the usual tangent bundle, one can characterize the failure of the PQP for the log tangent bundle via a maximal quotient.  However, we will not need this statement in this paper.

\subsection{Log free curves} 
Let $(X,\Delta)$ be a projective lc pair.  We define the space of log curves associated to $(X,\Delta)$ as follows:

\begin{definition} \label{defi:logcurve}
Let $(X,\Delta)$ be a projective lc pair and let $C$ be a smooth projective curve equipped with a reduced divisor $D$.  Suppose that $\{ \Delta_{i} \}_{i=1}^{k}$ are the irreducible components of $\lfloor \Delta \rfloor$.  For each index $i$ choose an effective divisor $D_{i} \subset C$ whose support is contained in $D$.  We define the space $\Mor^{log}_{\{ D_{i} \}}(C,X)$ to be the locally closed subscheme of $\Mor(C,X)$ consisting of maps $s: C \to X$ such that
\begin{enumerate}
\item the image $s(C)$ is contained in the SNC locus of $(X,\lfloor \Delta \rfloor)$ but not in $\Supp(\lfloor \Delta \rfloor)$, and
\item $s^{*}\Delta_{i} = D_{i}$ for every index $i$.
\end{enumerate}
We will say that a map $s: (C,D) \to (X,\Delta)$ is a log curve when we want to think of it as an element of the appropriate space $\Mor^{log}_{\{ D_{i} \}}(C,X)$.  We write $\Mor^{log}(C,X)$ for the union of the moduli spaces as we vary over all $\{D_{i}\}$.
\end{definition}

Let $V \subset X$ denote the open locus where the pair $(X,\lfloor \Delta \rfloor)$ is SNC.  We can equip $C \times V$ with the log structure defined by the SNC divisor $C \times \lfloor \Delta \rfloor|_{V}$.  By identifying a morphism $s: C \to V$ with its graph, one can identify $\Mor^{log}(C,X)$ as an open substack of the moduli space of log sections $\mathrm{Sec}_{log}(C \times V/C)$ as constructed in \cite[Section 3.1]{CLT25}.  According to \cite[Discussion after Equation (3.5)]{CLT25}, the log deformation theory of $s$ is controlled by the cohomology groups of the pullback of $T_{V}(-\log \lfloor  \Delta \rfloor )$, or equivalently, by the cohomology groups of $s^{*}T^{log}_{(X,\Delta)}$.

\begin{definition} \label{defi:rfreelogcurve}
Let $(X,\Delta)$ be a projective lc pair and let $C$ be a smooth projective curve of genus $g$.  Fix an $r \geq 0$.  A log curve $s: (C,D) \to (X,\Delta)$ is $r$-free if $\mu^{min}(s^{*}T_{(X,\Delta)}^{log}) \geq 2g(C) + r$. 

We say that $s: (C,D) \to (X,\Delta)$ is an unmixed $r$-free log curve if additionally for every pair of distinct irreducible components $\Delta_{i},\Delta_{j}$ of $\lfloor \Delta \rfloor$ the preimages $s^{*}\Delta_{i}$ and $s^{*}\Delta_{j}$ are disjoint.  (In other words, $s(C)$ is only allowed to meet the boundary $\lfloor \Delta \rfloor$ along codimension $1$ strata.)
\end{definition}

\begin{lemma} \label{lemm:maketransverse}
Let $(X,\Delta)$ be a projective lc pair over an algebraically closed field of characteristic $0$.  Suppose that $s: (C,D) \to (X,\Delta)$ is an $r$-free log curve.  Then there is a log resolution $\phi: Y \to X$ of $(X,\lfloor \Delta \rfloor)$ and an effective divisor $\Delta_{Y}$ on $Y$ such that $\phi: (Y,\Delta_{Y}) \to (X,\Delta)$ is log \'etale on a neighborhood of $s(C)$ and the strict transform of $s$ on $Y$ is an unmixed $r$-free curve with respect to $(Y,\Delta_{Y})$.
\end{lemma}

In characteristic $p$, one can accomplish the same goal simply by repeatedly blowing up strata of the boundary, but we will not need such a result.

\begin{proof}
For any log resolution $\phi: Y \to X$ of $(X,\lfloor \Delta \rfloor)$ we define $\Delta_{Y}$ to be the sum of the strict transform of $\lfloor \Delta \rfloor$ with the reduced $\phi$-exceptional divisor. 

Let $\phi: Y \to X$ be any log resolution of $(X,\lfloor \Delta \rfloor)$ that is an isomorphism over the SNC locus.  Thus the strict transform $s_{Y}$ of $s$ is $r$-free.  If $s_{Y}$ is not unmixed, then we let $\psi: \widetilde{Y} \to Y$ denote the blow-up of any minimal dimension strata of $\lfloor \Delta_{Y} \rfloor$  that meets every log deformation of $s_{Y}$.  Since this blow-up is log \'etale, the strict transform $s_{\widetilde{Y}}$ is again $r$-free.  After repeating this construction finitely many times we obtain the desired birational model.
\end{proof}

The following lemma verifies the basic properties of free log curves.

\begin{lemma} \label{lemm:logfreeprops}
Let $(X,\Delta)$ be a projective lc pair.  Suppose that $s: (C, D) \to (X, \Delta)$ is an $r$-free log curve with $r \geq 0$.
\begin{enumerate}
\item If $s$ is unmixed, then for any codimension $\geq 2$ closed subset $Z \subset X$ a general log deformation of $s$ will be disjoint from $Z$.
\item If $M \subset \Mor^{log}(C,X)$ is the irreducible component containing $s$ then the $(\lfloor r \rfloor + 1)$-fold evaluation map $\psi: C^{\times  (\lfloor r \rfloor   + 1)} \times M \to C^{\times  (\lfloor r \rfloor   + 1)} \times  X^{\times ( \lfloor r \rfloor   + 1)}$ is dominant. 
\end{enumerate}
\end{lemma}

\begin{proof}
We first explain (1). Observe that unmixed log free curves only meet the smooth locus of $\lfloor  \Delta \rfloor$. Let $M^{\circ}$ be the open locus in $M$ parametrizing $r$-free log curves.  For $x \in X$, consider the set $U_{p,x} \subset M^{\circ} \times C$ of pairs $(s,p)$ with $s(p)=x$.    For $x \in X^{sing} \cup \lfloor \Delta \rfloor^{sing}$, $U_{p,x}$ is empty.  For $x \in X \setminus \lfloor \Delta \rfloor$, $U_{p,x}$ has dimension at most $\dim M - \dim X$ because the map $U_{p,x} \to X$ is smooth due to the $r$-free assumption.  For $x \in \lfloor \Delta \rfloor^{sm}$, $U_{p,x}$ has dimension at most $\dim M - \dim X + 1$ because the map $T^{log}_{(X, \Delta)}|_p \to T_X|_p$ has image of codimension 1; however, $U_{p,x}$ will be empty if $p \notin D$.  Taking the union over all $p \in C$, we get a set $U_x$ of dimension at most $\dim M - \dim X + 1$. Taking the union over all $x \in Z$, we get a set $U_{Z}$ of dimension at most $\dim M - 1$. Since $U_Z$ does not dominate $M$, the result follows. 

To see (2), it suffices to find a single point $(s, p_1, \dots, p_{ \lfloor r \rfloor +1})$ where the map $d\psi$ is surjective.  In other words, it suffices to find $(s, p_1, \dots, p_{ \lfloor r \rfloor+1})$ with
\begin{equation*}
h^1(C,s^*T^{log}_{(X,\Delta)}(-p_1-\dots-p_{ \lfloor r \rfloor+1})) = 0.
\end{equation*}
The existence of such an $s$ follows immediately from the hypothesis that $s$ is $r$-free: Serre duality guarantees the vanishing of $h^{1}$ for any choice of $\lfloor r \rfloor + 1$ points on $C$.
\end{proof}

\subsection{Existence of log free curves}
The next result is the analogue of Theorem \ref{theo:maintheoremcharp} in the log setting.

\begin{theorem} \label{theo:logcharp}
Let $(X,\Delta)$ be a projective lc pair of dimension $n$ over an algebraically closed field.  Suppose that $H$ is an ample divisor such that $\mu^{min}_{H^{n-1}}(T^{log}_{(X,\Delta)}) > 0$.  Then there is a smooth projective curve $C$ and a finite set $D \subset C$ such that for every $r \geq 0$ there is an $r$-free log curve $s: (C,D) \to (X,\Delta)$. 
\end{theorem}

\begin{proof}
Let $i: C \to X$ denote a general curve obtained as a complete intersection of elements of $|m_{i}H|$ for a tuple of sufficiently positive integers $m_{i}$.  We may ensure that $i(C)$ will lie in the SNC locus of $(X,\Delta)$.  Since the log tangent bundle is a subsheaf of the tangent bundle, we also know that $\mu^{min}_{H^{n-1}}(T_{X}) > 0$.  Thus by a repeated application of \cite[1.2 Theorem]{Flenner84} (in characteristic $0$) or by applying Theorem \ref{theo:normalsemistability} and arguing as in Proposition \ref{prop:lccurveexists} (in characteristic $p$) we may also ensure that $i^{*}T_{(X,\Delta)}^{log}$ is ample.

In characteristic $0$, as in the proof of Theorem \ref{theo:positivetofree} we spread out to characteristic $p$ by constructing:
\begin{itemize}
\item a ring $R$ that is finitely generated over $\mathbb{Z}$,
\item a family $\mathcal{X} \to \Spec(R)$ equipped with a boundary $\Delta_{R}$ and an open locus $\mathcal{U}$, and
\item a smooth projective curve $\mathcal{C}$ over $\Spec(R)$ equipped with a morphism $i: \mathcal{C} \to \mathcal{X}$ and with a reduced divisor $\mathcal{D} \subset \mathcal{C}$ such that $i^{-1}(\lfloor \Delta_{R} \rfloor) \subset \mathcal{D}$.
\end{itemize}
After shrinking $\Spec(R)$ we may suppose that $i_{\mathfrak{p}}^{*}T_{(\mathcal{X}_{\mathfrak{p}},\Delta_{R,\mathfrak{p}})}^{log}$  is ample for every $\fp \in \Spec(R)$.  When $\fp$ is a maximal ideal of $R$, Lemma \ref{lemm:frobpullback} shows that by precomposing by an iterated Frobenius we can ensure that every positive rank quotient of the pullback of $T_{(\mathcal{X}_{\mathfrak{p}},\Delta_{R,\mathfrak{p}})}^{log}$ has slope at least $2g(C) + r$.  Since the residue field at $\fp$ is finite, after perhaps increasing the power of the Frobenius we obtain an $r$-free log curve  $s_{\mathfrak{p}}: \mathcal{C}_{\mathfrak{p}} \to \mathcal{X}_{\mathfrak{p}}$ such that $s_{\mathfrak{p}}^{-1}( \lfloor \Delta_{R,\mathfrak{p}} \rfloor)$ continues to lie in $\mathcal{D}_{\mathfrak{p}}$.  

Arguing as in \cite[p95]{Kollar}, one may construct a space of relative log maps $\Mor^{log}_{R}(\mathcal{C}, \mathcal{X})$ that is cut out from the relative morphism scheme $\Mor_{R}(\mathcal{C},\mathcal{X})$ by $\deg s_{\mathfrak{p}}^* \lfloor \Delta_{R,\mathfrak{p}} \rfloor$ equations in a neighborhood of $s_{\mathfrak{p}}$.  Thus any irreducible component $M$ of $\Mor^{log}_{R}(\mathcal{C}, \mathcal{X})$ that contains $s_{\fp}$ has dimension at least $H^{0}(\mathcal{C}_{\fp}, s_{\fp}'^{*}T^{log}_{(\mathcal{X}_{\fp},\Delta_{R,\fp})}) + \dim_{\fp}(\Spec(R))$  as in \cite[I.2.17 Theorem]{Kollar}.  Because $s_{\mathfrak{p}}$ is an $r$-free log curve, it follows that $\Mor^{log}(\mathcal{C}_{\mathfrak{p}}, \mathcal{X}_\mathfrak{p})$ must have the expected dimension at $s_{\mathfrak{p}}$, and hence $M$ must dominate $\Spec R$.  A general log deformation of $s_{\mathfrak{p}}$ in the generic fiber of $M$ yields an $r$-free log curve $s: (C,D) \to (X,\Delta)$ by \cite[Theorem 5]{Nitsure11}.

In characteristic $p$, arguing as in the proof of Theorem \ref{theo:positivetofree} we spread out to a finite field, precompose by a suitable iterated Frobenius, and then repeat the characteristic $0$ argument to find an $r$-free log curve $s$ on the generic fiber.
\end{proof}

In characteristic $0$, we prove the analogue of Theorem \ref{theo:lcandantinef} in the log setting.

\begin{theorem} \label{theo:maintheoremlog}
Let $(X,\Delta)$ be a projective $\mathbb{Q}$-factorial lc pair over an algebraically closed field of characteristic $0$.  Suppose that $-(K_{X} + \Delta)$ is nef.  Then the following conditions are equivalent:
\begin{enumerate}
\item The log tangent bundle admits no positive rank quotient with numerically trivial first Chern class.
\item $(X,\Delta)$ admits a $1$-free log curve.
\item For every $m > 0$ we have $h^{0}(V,\Omega_{V}(\log \lfloor \Delta \rfloor)^{\otimes m}) = 0$ where $V \subset X$ is the SNC locus of the pair $(X,\lfloor \Delta \rfloor)$.
\end{enumerate}

If these equivalent conditions hold, then there is a smooth projective curve $C$ and a finite set $D \subset C$ such that for every $r \geq 0$ there is an $r$-free log curve $s: (C,D) \to (X,\Delta)$. 
\end{theorem}

\begin{proof}
(1) $\implies$ (2): Under assumption (1),  $T_{(X,\Delta)}^{log}$ satisfies the PQP by Theorem \ref{theo:logtangentwqp}.  Thus for any ample divisor $H$ we have $\mu^{min}_{H^{n-1}}(T^{log}_{(X,\Delta)}) > 0$. We conclude the existence of $r$-free curves by Theorem \ref{theo:logcharp}.

(2) $\implies$ (3): Under assumption (2), for any general point $x \in V$ there is a curve $s: C \to V$ such that $s^{*}T_{(X,\Delta)}^{log} = s^{*}\Omega_{V}(\log \lfloor \Delta \rfloor)^{\vee}$ is ample.  Thus any section of $\Omega_{V}(\log \lfloor \Delta \rfloor)^{\otimes m}$ must vanish at $x$.  Since this is true for a dense set of points $x \in V$, we conclude that the space of global sections is $0$.

(3) $\implies$ (1): Let $\phi: Y \to X$ be a log resolution of $(X,\lfloor \Delta \rfloor)$ that is an isomorphism over $V$ and let $\Delta_{Y}$ be the sum of the strict transform of $\lfloor \Delta \rfloor$ with the reduced $\phi$-exceptional divisor.  Then we have $h^{0}(Y, \Omega_{Y}(\log \Delta_{Y})^{\otimes m}) = 0$ for every $m > 0$ as well.  Since $\Omega_{Y}$ is a subsheaf of $\Omega_{Y}(\log \Delta_{Y})$, we conclude that $\Pic^{0}_{Y/\mathbb{C}} = 0$, so every numerically trivial Cartier divisor on $Y$ is torsion.  By appealing to $\mathbb{Q}$-factoriality and pulling back divisors from $X$ to $Y$, we see the same is true of numerically trivial Weil divisors on $X$.

Suppose we had a surjection $s^{*}T_{(X,\Delta)}^{log} \to \mathcal{Q}$ where $c_{1}(\mathcal{Q})$ was numerically trivial.  Since a torsion sheaf always has pseudo-effective first Chern class, we may suppose $\mathcal{Q}$ is torsion-free.  Then for any $m > 0$ we have a chain of surjections on $V$
\begin{equation*}
T_{V}(- \log \lfloor \Delta \rfloor)^{\otimes m \rk(\mathcal{Q})} \to \mathcal{Q}|_{V}^{\otimes m \rk(\mathcal{Q})} \to \left(\bigwedge\nolimits^{\rk(\mathcal{Q})} \mathcal{Q}|_{X^{sm}} \right)^{\otimes m}.
\end{equation*}
In particular, when $m$ satisfies $mc_{1}(\mathcal{Q}) \sim 0$ we obtain a surjection $T_{V}(- \log \lfloor \Delta \rfloor)^{\otimes m\rk(\mathcal{Q})} \to \mathcal{O}_{V}$.  Dually, we get a non-vanishing tensor field on $V$, a contradiction.
\end{proof}

\begin{proof}[Proof of Theorem \ref{theo:logcase}:]
Follows immediately from Theorem \ref{theo:maintheoremlog}.
\end{proof}

\section{Finiteness of fundamental groups} \label{sect:fundgroups}

We next address finiteness of the fundamental group.  For the smooth locus of a klt Fano variety over $\mathbb{C}$, this famous question has been addressed in various settings -- the \'etale fundamental group, the topological fundamental group, and the regional fundamental group of a klt singularity -- by \cite{Xu14}, \cite{GKP16a}, \cite{BGO17}, \cite{TX17}, \cite{CRST18}, \cite{BCRGST19}, \cite{Braun21},  \cite{CarvajalStabler}.

In this section we use the existence of $1$-free curves to prove finiteness of fundamental groups.  (It was known previously that the existence of a $1$-free rational curve in $X^{sm}$ could be used to prove this statement -- see e.g.~\cite[7.5 Lemma]{KM99}.)
The proof relies on four properties of (topological or \'etale) fundamental groups over an algebraically closed field $k$.
\begin{enumerate}[label=(\Alph*)]
\item Suppose $X,Y$ are smooth varieties over $k$ with $X$ proper.  Given geometric points $x \in X$, $y \in Y$, the natural homomorphism $\pi_{1}(X \times Y,(x,y)) \to \pi_{1}(X,x) \times \pi_{1}(Y,y)$ is an isomorphism.  Furthermore the projection map is split by the inclusion of $\{x\} \times Y$.
\item Suppose $f: X \to Y$ is a dominant morphism of normal varieties over $k$.  For any $k$-point $x \in X$ the image $f_{*}\pi_{1}(X,x)$ has finite index in $\pi_{1}(Y,f(x))$.
\item We have $\pi_{1}(\Spec(k),\Spec(k)) = 0$.
\item Let $X$ be a variety over $k$.  If $x,x'$ are two $k$-points of $X$ then there is an isomorphism between $\pi_{1}(X,x)$ and $\pi_{1}(X,x')$ that is defined up to conjugacy.
\end{enumerate}
For topological fundamental groups, (A) is established by \cite[Proposition 1.12]{Hatcher02} and (B) is established by \cite[2.9 Proposition]{Kollar95b}.  For \'etale fundamental groups, (A) is established by \cite[Expos\'e X Corollaire 1.7]{SGA}.  The argument for (B) seems to be well-known: one can write $f$ as a composition of an open embedding, a morphism with geometrically connected fibers, and a generically \'etale map.  The first two maps induce surjections on \'etale fundamental groups and the last induces a map whose image has finite index.  Properties (C) and (D) are well-known in both settings.

\begin{theorem}  \label{theo:firstetalefundgroup}
Let $U$ be a smooth variety over an algebraically closed field $k$ and let $x \in U$ be a $k$-point.  Suppose that $C$ is a smooth projective $k$-curve, $M$ is a $k$-variety, and $e: C \times M \to U$ is a $k$-morphism such that the induced map
\begin{equation*}
r = (p_{1},p_{2},e \circ p_{13}, e \circ p_{23}): C \times C \times M \to C \times C \times U \times U
\end{equation*}
is dominant, where $p_{i}$ and $p_{ij}$ denote projections onto the corresponding factors.  Then:
\begin{enumerate}
\item The \'etale fundamental group $\pi_{1}^{et}(U,x)$ is finite.
\item If $k=\mathbb{C}$, the topological fundamental group $\pi_{1}(U,x)$ is finite.
\end{enumerate}
\end{theorem}

The proof incorporates a simplification recommended to us by Koll\'ar.

\begin{proof}
We start with (1). Fix a general $k$-point $y \in U$ and a general $k$-point $c \in C$.  Let $S \subset M$ be  the set of curves sending $c$ to $y$, i.e., $S = p_{3}(r^{-1}(C \times \{c\} \times U \times \{y\}))$.  By assumption, for a general $k$-point $z \in C$ the map $e: C \times S \to U$ induces a dominant map $e: \{ z \} \times S \to U$.  We may replace $S$ by a smooth open subset of a suitably chosen irreducible component while still keeping the dominance of $e: \{z\} \times S \to U$.

By (D) there is an isomorphism $\pi_{1}^{et}(C,c) \cong \pi_{1}^{et}(C,z)$ which is determined up to an inner automorphism.  Let $s$ be any $k$-point of $S$.  Appealing to (A) we obtain an isomorphism
$$\pi_{1}^{et}(C \times S, c \times s) \cong \big( \pi_{1}^{et}(C,c) \times \pi_{1}^{et}(S,s) \big ) \cong \big (\pi_{1}^{et}(C,z) \times \pi_{1}^{et}(S,s)\big ) \cong \pi_{1}^{et}(C \times S, z \times s)$$ 
that is determined up to conjugacy on the first factor of the middle terms.  Using the splitting property in (A), we conclude that $e_{*}\pi^{et}_{1}(\{ z \}  \times S,z \times s)$ in $\pi_{1}^{et}(U,e(z \times s))$ is conjugate to $e_{*}\pi^{et}_{1}(\{c\} \times S, c \times s)$ in $\pi^{et}_{1}(U,y)$.  Since $e$ contracts $\{c\} \times S$, by (C) the latter image is trivial.  On the other hand, by (B) the subgroup $e_{*}\pi^{et}_{1}(\{z\} \times S, z \times s) \subset \pi^{et}_{1}(U,e(z \times s))$ has finite index.  We conclude that $\pi_{1}^{et}(U,e(z \times s))$ is finite; the claim for arbitrary $x \in U$ follows from (D).  

The proof of (2) is the same as the proof of (1).
\end{proof}

\begin{proof}[Proof of Theorem \ref{theo:firstfundgroups}:]
Theorem \ref{theo:lcandantinef} proves the existence of a $1$-free curve $s: C \to X^{sm}$.  Proposition \ref{prop:freeproperties}.(2) shows that deformations of $s$ go through two general points of $X^{sm}$.  We conclude the desired statement by Theorem \ref{theo:firstetalefundgroup}.(2).
\end{proof}

\begin{proof}[Proof of Theorem \ref{theo:etalefundgroup}:]
Follows from Proposition \ref{prop:freeproperties}.(2) and Theorem \ref{theo:firstetalefundgroup}.(1).
\end{proof}

\subsection{Log setting}
We next prove a finiteness statement for fundamental groups in the log setting.  Note that in characteristic $p$ even the simplest example $U = \mathbb{A}^{1} = \mathbb{P}^{1} \backslash \{\infty\}$ will have infinite \'etale fundamental group.  We must instead work with the curve-tame \'etale fundamental group as defined in \cite{Wiesend08}, \cite{KS10}.  We write $\pi_{1}^{t}(X,x) := \pi_{1}^{t}(X/\Spec(k),x)$ for this group.

The following statement is almost identical to Theorem \ref{theo:firstetalefundgroup}.  The key difference is that the curve $C$ is only quasiprojective, since in the log setting we can no longer hope to find projective curves in the open locus $U$.  The proof relies on the results of Appendix \ref{sect:appendix}.

\begin{theorem}  \label{theo:secondetalefundgroup}
Let $U$ be a smooth variety over an algebraically closed field $k$ and let $x \in U$ be a $k$-point.  Suppose that $C$ is a smooth irreducible quasi-projective $k$-curve, $M$ is a $k$-variety, and $e: C \times M \to U$ is a $k$-morphism such that the induced map
\begin{equation*}
r = (p_{1},p_{2},e \circ p_{13}, e \circ p_{23}): C \times C \times M \to C \times C \times U \times U
\end{equation*}
is dominant, where $p_{i}$ and $p_{ij}$ denote projections onto the corresponding factors.
\begin{enumerate}
\item The tame fundamental group $\pi_{1}^{t}(U,x)$ is finite.
\item If $k=\mathbb{C}$, then the topological fundamental group $\pi_{1}(U,x)$ is finite.
\end{enumerate}
\end{theorem}

\begin{proof}
(1) We have the following properties of tame fundamental groups. The only difference from the corresponding properties for \'etale fundamental groups is that no properness assumption is necessary in (A).
\begin{enumerate}[label=(\Alph*)]
\item Suppose $X,Y$ are smooth varieties over $k$.  Given geometric points $x \in X$, $y \in Y$, Theorem \ref{theo:mainappendix} shows that the natural homomorphism $\pi_{1}^{t}(X \times Y,(x,y)) \to \pi_{1}^{t}(X,x) \times \pi_{1}^{t}(Y,y)$ is an isomorphism.  Furthermore the projection map is split by the inclusion of $\{x\} \times Y$.
\item Suppose $f: X \to Y$ is a dominant morphism of normal varieties over $k$.  Then for any $k$-point $x \in X$, $f_{*}\pi_{1}^{t}(X,x)$ has finite index in $\pi_{1}^{t}(Y,f(x))$.  This follows from the corresponding fact for \'etale fundamental groups via the existence of a functorial surjection $\pi_{1}^{et}(X,x) \to \pi_{1}^{t}(X,x)$.
\item We have $\pi_{1}^{t}(\Spec(k),\Spec(k)) = 0$.
\item Let $X$ be a variety over $k$.  If $x,x'$ are two $k$-points of $X$ then there is an isomorphism between $\pi^{t}_{1}(X,x)$ and $\pi^{t}_{1}(X,x')$ that is defined up to conjugacy.
\end{enumerate}
Using these properties, the proof is identical to the proof of Theorem \ref{theo:firstetalefundgroup}.(1).

(2) Follows from the same argument as Theorem \ref{theo:firstetalefundgroup}.(2).
\end{proof}

\begin{proof}[Proof of Theorem \ref{theo:logfundgroup}:]
Let $U = X \backslash \lfloor \Delta \rfloor$.  Let $M \subset \Mor^{log}_{\{D_{i}\}}(C,X)$ denote an irreducible component containing a $1$-free log curve $s: (C,D) \to (X,\Delta)$ and let $e: C \times M \to X$ denote the evaluation map.  Lemma \ref{lemm:logfreeprops}.(2) shows that the induced map $C^{\times 2} \times M \to C^{\times 2} \times X^{\times 2}$ is dominant.

We let $C^{\circ} \subset C$ be the complement of $D$ so that we have a map $e: C^{\circ} \times M \to U^{sm}$.  It is clear that $C^{\circ \times 2} \times M \to C^{\circ \times 2} \times (U^{sm})^{ \times 2}$ is still dominant.  We conclude by applying Theorem \ref{theo:secondetalefundgroup}.
\end{proof}

\subsection{Dichotomy} \label{sect:dichotomy}
Theorem \ref{theo:lcandantinef} and Lemma \ref{lemm:trivquotient} establish a fundamental geometric dichotomy for $\mathbb{Q}$-factorial lc Fano pairs $(X,\Delta)$ over $\mathbb{C}$.  When $T_{X}$ satisfies the PQP, $X$ also satisfies other desirable geometric properties:
\begin{enumerate}
\item $X$ carries a $1$-free curve by Theorem \ref{theo:lcandantinef},
\item $X^{sm}$ has no non-vanishing tensor fields by Theorem \ref{theo:lcandantinef},
\item $X$ is rationally connected by \cite[Theorem 5.2]{Gounelas16}, and
\item the smooth locus of $X$ has finite fundamental group by  Theorem \ref{theo:firstfundgroups}.
\end{enumerate}

On the other hand, when $T_{X}$ does not satisfy the PQP then Lemma \ref{lemm:trivquotient} shows that there is a canonical quotient $\pi : T_{X} \to \mathcal{Q}$ witnessing this failure where $c_1(\mathcal{Q})$ is numerically trivial.  The kernel of $\pi$ is an algebraically integrable foliation and every almost free curve on $X$ is contained in a fiber of the corresponding rational map.  Indeed, as an almost free curve $s: C \to T_{X}$ can be moved off of any codimension $2$ subset, we may assume $s(C)$ is contained in the smooth locus of the map induced by $T_{X} \to \mathcal{Q}$.  Only the trivial factors in $s^{*}T_{X}$ can admit non-trivial maps to $s^{*}\mathcal{Q}$, showing that $s(C)$ must be contained in a fiber.

The following example clarifies that the dichotomy discussed above is not controlled by the rationally connected property or by the fundamental group.

\begin{example} \label{exam:pqpfailsexample}
We give an example of an lc Fano variety $X$ such that $X$ is rationally connected, $X^{sm}$ has finite fundamental group, and $T_{X}$ fails the PQP.  (In other words, the PQP for $T_{X}$ implies the first two properties but is not necessary for them to hold.)

Let $B$ be a smooth sextic in $\mathbb{P}^{2}$.  Starting from $\mathbb{P}^{2} \times \mathbb{P}^{1}$, we choose a section $S = \mathbb{P}^{2} \times \{ 0 \}$ and define $F = B \times \mathbb{P}^{1}$.  We then perform the following blow-ups:
\begin{itemize}
\item blow up $S \cap F$ to obtain an exceptional divisor $E_{1} \cong B \times \mathbb{P}^{1}$;
\item blow up the intersection of $E_{1}$ with the strict transform of $F$ to get an exceptional divisor $E_{2}$.
\end{itemize}
The resulting variety $\widetilde{X}$ is equipped with a morphism $\widetilde{\pi}: \widetilde{X} \to \mathbb{P}^{2}$ obtained by composing the blow-downs and the projection map.  We denote by $\widetilde{S}, \widetilde{F}, \widetilde{E}_{1}$ the strict transforms on $\widetilde{X}$ of $S,F,E_{1}$.  We also denote by $H_{1}$ and $H_{2}$ the pullbacks to $\widetilde{X}$ of $\mathcal{O}(1)$ from $\mathbb{P}^{2}$ and $\mathbb{P}^{1}$ respectively.  Every fiber of $\pi$ over $B$ consists of three components: a reduced component $\ell_{\widetilde{F}}$ lying in $\widetilde{F}$, a reduced component $\ell_{\widetilde{E}_{1}}$ lying in $\widetilde{E}_{1}$ and a multiplicity two component $\ell_{E_{2}}$ lying in $E_2$. 

We next construct a  variety $X$ by contracting certain divisors on $\widetilde{X}$.  Define the divisor $D = 6H_{1} + 2H_{2} - \widetilde{E}_{1} - 2E_{2}$.  Note that $D$ can also be expressed as $D = \widetilde{F} + 2H_{2}$ and as $D  = 2\widetilde{S} + \widetilde{E}_{1} + 6H_{1}$.  From these expressions we see that $D$ is big and nef and that the base locus of $|D|$ is empty.  Since $D$ has vanishing intersection against $\ell_{\widetilde{F}}$, $\ell_{\widetilde{E}_{1}}$, and every curve in $\widetilde{S}$, the linear series $|D|$ defines a birational morphism $\phi: \widetilde{X} \to X$ that contracts the divisors $\widetilde{F}, \widetilde{E}_{1}, \widetilde{S}$.  In particular $X$ has Picard rank $1$.  Since $\phi^{*}K_{X} = K_{\widetilde{X}} + \widetilde{S} + \frac{1}{2}\widetilde{E}_{1}$ we see that $X$ is a log Fano variety with log canonical singularities.

Note that $\widetilde{\pi}: \tilde{X} \to \PP^2$ induces a rational map $\pi: X \dashrightarrow \mathbb{P}^{2}$ which is a morphism on $X^{sm}$. We claim that this map yields a quotient $T_{X} \to \mathcal{Q}$ such that $c_{1}(\mathcal{Q})$ is numerically trivial.  Indeed, consider the corresponding map $\widetilde{\pi}: \widetilde{X} \to \mathbb{P}^{2}$.  Since the fibers over $B$ have multiplicity $2$ along $E_{2}$, the image of the map $T_{\widetilde{X}} \to \widetilde{\pi}^{*}T_{\mathbb{P}^{2}}$ has first Chern class $3H_{1} - E_{2}$.  Since this divisor is $\mathbb{Q}$-linearly equivalent to $\frac{1}{2}(\widetilde{F} + \widetilde{E}_{1})$, by pushing forward to $X$ we see that $\mathcal{Q}$ has numerically trivial first Chern class.

The previous paragraph showed that $T_{X}$ does not have the PQP.  It is clear that $X$ is rationally connected.  It only remains to verify that $X^{sm}$ has finite fundamental group.  The double cover of $\mathbb{P}^{2}$ ramified over $B$ is a K3 surface $Z$.  Let $U$ denote the normalization of the base change $Z \times_{\mathbb{P}^{2}} X^{sm}$.  Since every fiber of $X^{sm}$ over $B$ has multiplicity $2$, $U \to X^{sm}$ is an \'etale double cover.  Furthermore, $U$ is an open subset of an $\mathbb{A}^{1}$-bundle $Y$ over $Z$ where the complement of $U$ in $Y$ has codimension $\geq 2$.  This implies that $\pi_{1}(U)$, and hence also $\pi_{1}(X^{sm})$, is finite. 
\end{example}

\printbibliography

\appendix


\begin{refsection}
\section{Tame fundamental group of a product \\ (written by Aise Johan de Jong)} \label{sect:appendix}

Setup: we take an algebraically closed field $k$ and we set $S = \Spec(k)$.
When $X$ is a smooth variety over $k$ we set
$$
\pi_1^t(X) = \pi_1^t(X/S, \overline{x})
$$
for some choice of geometric point $\overline{x}$ of $X$
as defined in \cite[Section 7]{KS10}. Correspondingly
we will say a finite \'etale cover of $X$ is {\it tame} if it is curve-tame
in the sense of \cite{KS10}.

\medskip\noindent
If $Y$ is a second smooth variety over $k$, then we denote by $X \times Y$
the product variety. Choose closed points $x \in X$, $y \in Y$. We will
use $x$, $y$, and $(x, y)$ as the base points of our tame fundamental groups.
Functoriality of the tame fundamental group
(discussed in \cite{KS10}) gives homomorphisms
$$
\pi_1^t(X) \to \pi_1^t(X \times Y)
\quad\text{and}\quad
\pi_1^t(Y) \to \pi_1^t(X \times Y)
$$
It turns out that we also have a canonical surjective map
$$
\pi_1^t(X \times Y) \to
\pi_1^t(X) \times \pi_1^t(Y)
\leqno{(*)}
$$
corresponding to the fully faithful functor
$$
\textit{F\'Et}_X \times \textit{F\'Et}_Y \to
\textit{F\'Et}_{X \times Y}, (U, V) \mapsto U \times V
$$
of categories of (tame) finite \'etale coverings.
This looks strange at first sight; in terms of Galois groups, it
corresponds to the fact that the absolute Galois group of the
function field $k(X \times Y)$ maps onto the product of the
absolute Galois groups of $k(X)$ and $k(Y)$.

\begin{theorem} \label{theo:mainappendix}
The map (*) is an isomorphism.
\end{theorem}

\cite[Proposition 3]{Hoshi09} gives a related but slightly weaker statement in a different setting.

\begin{proof}
We will reduce this theorem to Lemma \ref{lemma-curves}.

\medskip\noindent
We claim that the images of $\pi_1^t(X)$ and $\pi_1^t(Y)$ in
$\pi_1^t(X \times Y)$ commute. By Lemma \ref{lemma-curves}
this is true for elements in the image of $\pi_1^t(C) \to \pi_1^t(X)$
and $\pi_1^t(D) \to \pi_1^t(Y)$ for any pair of morphisms
$f : C \to X$ and $g : D \to Y$ where $C$ and $D$ are smooth
curves and $x \in f(C)$ and $y \in g(D)$. By Lemma \ref{lemma-lefschetz}
such elements form dense subgroups of $\pi_{1}^t(X)$ and $\pi_{1}^t(Y)$ and this proves the claim.  Thus we have an induced map $\pi_1^t(X) \times \pi_1^t(Y) \to \pi_1^t(X \times Y)$.  It is clear that this map splits the surjection $\pi_1^t(X \times Y) \to
\pi_1^t(X) \times \pi_1^t(Y)$ discussed earlier.

\medskip\noindent
It suffices to show that the map $\pi_1^t(X) \times \pi_1^t(Y) \to \pi_1^t(X \times Y)$ is surjective.
Suppose that $Z \to X \times Y$ is a connected tame Galois cover.
By Lemma \ref{lemma-lefschetz} we can choose a
morphism $g : C \to X \times Y$ where $C$ is a smooth curve
with $(x, y) \in g(C)$ such that $Z \times_{X \times Y} C$ is connected.
Let $C_1$, resp.\ $C_2$ be the normalization of the image of
$C$ in $X$, resp.\ $Y$.  Since $g$ factors through $C_{1} \times C_{2}$, the covering
$Z \times_{X \times Y} (C_1 \times C_2)$ is also connected.  By Lemma \ref{lemma-curves}.
we see that $Z$ defines a connected $\pi_1^t(C_1) \times \pi_1^t(C_2)$ set.
As the action of these groups on Z factors through $\pi_1^t(X) \times \pi_1^t(Y)$, we obtain the claim.
\end{proof}

\begin{lemma}
\label{lemma-curves}
If $X$ and $Y$ are curves, then (*) is an isomorphism.
\end{lemma}

\begin{proof}
Choose smooth projective compactifications $X \subset \bar X$ and
$Y \subset \bar Y$. Let $f : Z \to X \times Y$ be a tame finite \'etale
Galois covering with group $G$ (we do not require connectedness
in this proof).  The covering
$X' = f^{-1}(X \times \{y\}) \to X$ is a tame finite \'etale Galois
covering with group $G$.
Similarly for $Y' = f^{-1}({x} \times Y) \to Y$.
The ramification index of $X' \to X$ at a point
$x_1 \in \overline{X} \setminus X$ is the same as the ramification
index of $Z \to X \times Y$ along the divisor
$\{x_1\} \times \bar Y \subset \bar X \times \bar Y$; here we use
the Galois structure to know that all the ramification indices
above a given point are the same. Similarly for $Y'$.
Thus Abhyankar's lemma shows that the pull back $Z'$ of $Z$
to $X' \times Y'$ does not ramify at all along the generic
points of the boundary of $X' \times Y'$ in $\bar X' \times \bar Y'$.
Here $\bar X'$ and $\bar Y'$ are the smooth projective compactifications of
$X'$ and $Y'$; of course $\bar X' \times \bar Y'$
is the normalization of $\bar X \times \bar Y$ in  the function field of $X' \times Y'$. 
By purity we see that $Z'$ extends to a finite \'etale covering
of $\bar X' \times \bar Y'$. Thus $Z'$ is dominated by a product
$X'' \times Y''$ with $X''$ and $Y''$ tame finite \'etale over $X$ and $Y$.
This proves that the map (*) is injective, whence an isomorphism.
\end{proof}

\begin{lemma}[\cite{EK16}]
\label{lemma-lefschetz}
Let $X$ be a smooth variety over $k$ and $x \in X$ be a closed point.
Let $Z \to X$ be finite tame \'etale with $Z$ connected.
Then there exists a morphism $f : C \to X$ where $C$ is a smooth
curve with $x \in f(C)$ such that $Z \times_X C$ is connected.
\end{lemma}

\begin{proof}
Since $X$ and thus $Z$ are smooth, $Z$ remains connected after passing to a non-empty open subset.
Thus we may assume that $X$ is quasiprojective.  Furthermore if $X'$ denotes the blow-up of $X$ at $x$ then the pull back of $Z$ to $X'$ will remain connected.  Let $\bar{X}'$ denote a normal projective closure of $X'$ and let $C \subset \bar{X}'$ be a smooth complete intersection curve passing through $x$ that is transverse to $D := \bar{X'}\setminus X'$ and disjoint from the singular locus of $D$ and $\bar{X}'$.  By \cite[Proposition 6.2]{EK16} (see \cite{EK16err}), the pull back of $Z$ to $C$ is connected.
\end{proof}

\printbibliography[title={References (for Appendix~\ref{sect:appendix})}]

\end{refsection}

\end{document}